\documentclass[final]{siamltex}
\usepackage{amsfonts}
\usepackage{graphics}
\usepackage{caption2}
\usepackage{amssymb}
\usepackage{psfrag}
\usepackage{amsmath}
\usepackage{graphicx}
\usepackage{multirow}
\usepackage{cite}
\usepackage{subfigure}
\usepackage[normalem]{ulem}
\usepackage{color}

\allowdisplaybreaks[4]

\newcommand{\bm}[1]{\mbox{\boldmath{$#1$}}}

\def\x{{\bf   x}}

\def\({\left(}
\def\[{\left[}
\def\){\right)}
\def\]{\right]}

\def\grad{\nabla }

\newtheorem{lem}{Lemma}
\newtheorem{thm}{Theorem}

\numberwithin{equation}{section}
\numberwithin{thm}{section}
\numberwithin{lem}{section}

\begin{document}
\title{A novel energy factorization  approach for the diffuse-interface model with Peng-Robinson equation of state
}

\author{Jisheng Kou\thanks{School of Civil Engineering, Shaoxing University, Shaoxing 312000, Zhejiang, China; School of Mathematics and Statistics, Hubei Engineering  University, Xiaogan 432000, Hubei, China. Email: {\tt jishengkou@163.com}.} 
\and Shuyu Sun\thanks{Corresponding author. Computational Transport Phenomena Laboratory, Division of Physical Science and Engineering,
King Abdullah University of Science and
Technology, Thuwal 23955-6900, Kingdom of Saudi Arabia.   Email: {\tt shuyu.sun@kaust.edu.sa}.}
\and  Xiuhua Wang\thanks{School of Mathematics and Statistics, Hubei Engineering  University, Xiaogan 432000, Hubei, China.}
}

 \maketitle

\begin{abstract}
The Peng-Robinson equation of state (PR-EoS) has become one of the most  extensively applied   equations of state    in   chemical engineering and petroleum industry due to its excellent accuracy   in predicting the  thermodynamic properties of a wide variety of  materials, especially   hydrocarbons. Although great efforts   have been made to construct efficient numerical methods for the diffuse interface models with PR-EoS,   there is still not a linear numerical scheme that can be proved to preserve the original energy dissipation law.  
In order to pursue  such a  numerical scheme,  we  propose a novel energy factorization (EF) approach, which first factorizes an energy function  into   a product of several factors and then treats the factors  using their  properties to obtain the semi-implicit linear schemes.  
We apply the EF approach to deal with the   Helmholtz free energy density   determined by PR-EoS, and then propose a linear  semi-implicit numerical scheme that   inherits the original energy dissipation law.  Moreover,  the proposed   scheme is proved to satisfy  the maximum principle in both the time semi-discrete   form and the cell-centered finite difference fully discrete form under certain conditions. Numerical results are presented to demonstrate the stability and efficiency of the proposed scheme. 
 
\end{abstract}

\section{Introduction}

The Peng-Robinson equation of state (PR-EoS)  \cite{Peng1976EOS} has become one of the most popular   and useful   tools for     describing the thermodynamic properties of fluids   in both academic  and industrial fields, especially chemical engineering and petroleum industry \cite{PREOS2017}.
Compared to the well-known Van der Waals equation of state,  PR-EoS can provide   more reasonable accuracy   in predicting the   properties of a wide variety of  materials, such as N$_2$, CO$_2$ and  hydrocarbons. 
PR-EoS has been extensively   applied   for simulation of many important   problems in petroleum  and chemical engineering, for instance,    phase equilibria calculations \cite{jindrova2013fast,mikyvska2015General,mikyvska2011new,kousun2018Flash,Nagarajan1991,firoozabadi1999thermodynamics,mikyvska2018Flash,michelsen1999} and prediction of surface tension between gas and liquid  \cite{miqueu2004modelling,kousun2015CMA,kousun2015SISC,firoozabadi1999thermodynamics}. 
Modeling and simulation of compressible multi-component  two-phase flows with partial miscibility and realistic equations of state (e.g. PR-EoS)  are intensively  studied  in recent years  \cite{qiaosun2014,kousun2015SISC,kousun2015CHE,fan2017componentwise,kou2018nonisothermal,kou2017nonisothermal,kou2018SAV,Li2017IEQ,Peng2017convexsplitting,kouandsun2016multiscale}.     On the basis of  the thermodynamic fundamental  laws and  realistic equations of state  (e.g. PR-EoS),   general diffuse interface models for compressible multi-component  two-phase flows   have been proposed in \cite{kouandsun2017modeling} for isothermal fluids and \cite{kou2018nonisothermal,kou2017nonisothermal} for non-isothermal fluids.   

 This paper is primarily concerned with efficient numerical methods for an isothermal diffuse interface model with PR-EoS. There exist two  primary challenging problems  in designing   numerical schemes for such   model.  One is that  the   Helmholtz free energy density determined by PR-EoS has the  complicated structures and  strong  nonlinearity.  
The other is  that the model follows the energy dissipation law, and  thus numerical schemes should be constructed to preserve this feature at the discrete level.  In this paper, we will resolve the above challenging problems and will develop a  novel  linear, energy stable numerical scheme.

We now provide the up-to-date review regarding the approaches in the literature  employed to  handle  the bulk Helmholtz free energy density of PR-EoS and design  energy stable numerical schemes. One approach is the convex splitting method \cite{Elliott1993ConvexSpliting,Eyre1998ConvexSplitting} that has been  extensively used for various  phase-field models     \cite{shen2015SIAM,Wise2009Convex,Baskaran2013convexsplitting}.  For  the diffuse-interface models with PR-EoS,    the convex splitting schemes inheriting the discrete  energy dissipation law  have been developed   in a series of recent works \cite{qiaosun2014,fan2017componentwise,kouandsun2017modeling,kousun2018Flash,Peng2017convexsplitting,kou2017nonisothermal}.    
 When the convex splitting approach  is applied to the PR-EoS based Helmholtz free energy density, the ideal and repulsion terms are usually  treated implicitly due to their convexity, while the attraction  term with the concavity  is treated explicitly.  The convex splitting approach can produce unconditionally energy stable numerical schemes, but it results in  the nonlinear discrete equations that demand  the complicated implement  of efficient nonlinear iterative solvers  and also cost expensively in the computations.
Approximating    chemical potential by a difference of Helmholtz free energy density,    a fully-implicit  scheme  has been developed in \cite{kousun2015CHE}.  This scheme is   proved to be unconditionally energy stable, but it still suffers from the nonlinearity  of the resulted equations. 

The invariant energy quadratization (IEQ) approach   \cite{Yang2019IEQ,Yang2017IEQ,Yang2017IEQ2} is  a novel and efficient method developed in recent years  and  has been intensively applied  for various  phase-field models. The essential idea of   IEQ        is to transform the bulk free energy  into a quadratic form through introducing   a set of new variables. The new variables can be updated with  time steps via the semi-implicit linear schemes. A  numerical scheme has been developed in \cite{Li2017IEQ}  applying the IEQ approach to PR-EoS. As a   modification of the IEQ approach, the scalar auxiliary variable (SAV) approach has been proposed in \cite{Shen2018SAV}, which introduces a scalar auxiliary variable instead of the space-dependent new variables   in the IEQ approach.  It leads to unconditionally stable numerical schemes, which   only need to solve the linear equations with constant coefficients  at each time step. Recently,  in \cite{kou2018SAV},   the SAV approach  has been applied to construct  unconditionally energy stable numerical schemes for the model proposed in \cite{kouandsun2017modeling}, and moreover,   a component-wise SAV approach has also been developed.  The  numerical schemes constructed by the  IEQ  and SAV approaches are  linear, and consequently   are easy-to-implememt   and very  efficient in computations.   However, while   IEQ and SAV  have become the very useful and successful     tools applied for a variety of phase-field models, 
the produced  schemes use the    transformed  free energies that are generally never equivalent to the original  energies at the discrete level. Indeed,  it has been indicated in \cite{Yang2019IEQ} that  the transformed free energies have   the errors of  the order of time step size against  the original  energies. 
Consequently,  the   schemes constructed by  IEQ and SAV may  not inherit the original energy dissipation law  in theory  although the dissipation of  transformed energies can be proved.   In numerical tests of \cite{kou2018SAV},    the original   energy instability    has   been observed despite the transformed    energies   decrease with time steps.  
To our best knowledge, for the diffuse interface model with PR-EoS, there is not a linear   semi-implicit numerical scheme inheriting  the original energy dissipation law  so far. In this paper, we will propose such a    scheme. 

In this paper, we will propose a novel energy factorization (EF) approach to construct an efficient numerical scheme for the diffuse interface model with PR-EoS. 
  The key  idea of EF  is that we first factorize an energy function  into   a product of several factors and then handle them  by the use of their properties to obtain the semi-implicit linear schemes. The  EF approach will not introduce any new independent energy variable, thus it can ensure the original  energy dissipation law. 
Applying the EF approach to the model with PR-EoS, we will obtain a linear, efficient  semi-implicit numerical scheme that  inherits the original energy dissipation law.  

We note that molar density is the primal variable in the models with PR-EoS.  Moreover, for a realistic substance, molar density shall  have the physical limits under given thermodynamical conditions, and as a result, the maximum principle is essential for numerical methods to ensure that numerical solutions are physically reasonable.  However, there are no results concerning  the maximum principle of numerical schemes for the models with PR-EoS so far.  For the first time, the proposed numerical scheme will be proved to preserve the maximum principle under certain conditions. The proof will be provided for both the semi-discrete time matching scheme and the cell-centered finite difference fully discrete scheme.

The rest of this paper is organized as   follows. In Section 2, we will provide  a brief description of the diffuse interface model with PR-EoS. In Section 3, we will propose the energy factorization approach to deal with the  bulk Helmholtz free energy density. 
 In Section 4,  we will present the semi-discrete time  scheme and prove some theoretical results including the  energy stability and the maximum principle. The fully discrete schemes will be developed  and analyzed in Section 5.  In Section 6,  numerical results will be presented  to validate  the proposed numerical scheme.   Finally,   some concluding remarks are given in Section 7.

\section{Model equations}
We now give a brief description for the   Helmholtz free energy density of a bulk fluid (denoted by $f_{b}$)   determined by Peng-Robinson equation of state \cite{Peng1976EOS}.  We denote by $c$ molar density of a substance. For   specified temperature $T$,   $f_{b}$  is a function of molar density $c$  and can be expressed as a sum of three contributions
\begin{eqnarray}\label{eqBulkHelmholtzEnergy}
    f_b(c)&=& f_b^{\textnormal{ideal}}(c) + f_b^{\textnormal{repulsion}}(c)+f_b^{\textnormal{attraction}}(c),
\end{eqnarray}
where
\begin{eqnarray}\label{eqBulkHelmholtzEnergyIdeal}
    f_b^{\textnormal{ideal}}(c)&=&  c\vartheta_0 +  cRT\ln\(c\),
\end{eqnarray}
\begin{eqnarray}\label{eqBulkHelmholtzEnergyRepulsion}
    f_b^{\textnormal{repulsion}}(c)=-cRT\ln\(1-\beta c\),
\end{eqnarray}
\begin{eqnarray}\label{eqBulkHelmholtzEnergyAttraction}
    f_b^{\textnormal{attraction}}(c)= \frac{\alpha(T)c}{2\sqrt{2}\beta}\ln\(\frac{1+(1-\sqrt{2})\beta c}{1+(1+\sqrt{2})\beta c}\).
\end{eqnarray}
Here, $R$ is the universal gas constant and $\vartheta_0$ is an energy parameter that relies on   the temperature and thermodynamical properties of  a specific substance.   
The substance-specific   parameters $\alpha$ and $\beta$ can be determined from the critical properties and acentric factor of a specific substance    
\begin{equation}\label{eqaibi}
   \alpha= 0.45724\frac{R^2T_{c}^2}{P_{c}}\[1+m(1-\sqrt{T_{r}})\]^2,~~~~\beta= 0.07780\frac{RT_{c}}{P_{c}},
\end{equation}
where  $T_{r}=T/T_{c}$ is the   reduced temperature,     $T_{c}$ and  $P_{c}$ stands for  the   critical temperature and critical pressure respectively, and   $m$ is calculated  from the acentric factor  $\omega$ as follows
 \begin{subequations}\label{eqdefmi}
\begin{equation*}
 m=0.37464 + 1.54226\omega-  0.26992\omega^2 ,~~\omega\leq0.49,
\end{equation*}
\begin{equation*}
 m=0.379642+1.485030\omega-0.164423\omega^2 +0.016666 \omega^3,~~\omega>0.49.
\end{equation*}
\end{subequations}
  
From the physical point of view, $\beta$ is the   effective volume of one mole of a substance. Let $v=\frac{1}{c}$ be molar volume, i.e., the average volume occupied by one mole of a substance, which includes the effective volume $\beta$ and the space between molecules.  Therefore, we know that generally $\beta\ll v$ for gas and liquid. For an ideal and simple example, if the molecules can be approximated as spherical particles,  then we have
\begin{equation*}
 \beta c=\frac{\beta}{v}\leq\frac{\pi}{6} .
\end{equation*}  
This fact suggests that $\beta c$ has an upper bound for a specific substance under specified temperature. 
As a matter of fact,  from the following form of  Peng-Robinson equation of state (PR-EOS) \cite{Peng1976EOS}  
\begin{eqnarray} \label{eqPREOS}
    P = \frac{cRT}{1-\beta c}-\frac{\alpha(T)c^2}{1+\beta c+\beta c(1-\beta c)}> \frac{cRT}{1-\beta c}-\alpha(T)c^2,
\end{eqnarray}
where $P$ is the pressure, we can directly deduce 
\begin{eqnarray*} 
    \beta c <1- \frac{cRT}{P+\alpha(T)c^2}.
\end{eqnarray*}
To justify the boundedness  of $\beta c$ in realistic cases, we consider the species of n-butane. We can calculate  $\beta=7.2381\times10^{-5}$ m$^3$/mol using \eqref{eqaibi} and the physical data of n-butane.   Molar densities of gas and liquid of  n-butane at the temperature   330 K and pressure 106.39 bar are $c^G=249.1123$ mol/m$^3$ and $c^L= 9526.8428$ mol/m$^3$ respectively. Then we have
\begin{eqnarray*} 
    \beta c^G =0.0180,~~~\beta c^L =0.6896.
\end{eqnarray*}
 On the basis of  the above physical observation, we assume that molar density $c$ is always bounded as below
\begin{equation}\label{eqMolDensRange}
 0<c_m\leq c\leq c_M,~~\beta c_M\leq\epsilon_0<1 ,
\end{equation}
where $c_m$, $c_M$ and $\epsilon_0$  are determined by a specific substance under specified pressure and temperature. 
We remark that $\epsilon_0$  usually does not take   some very small value from physical property as stated above.

Since the diffuse  interfaces always exist between gas and liquid phases. In addition to the bulk free energy density,  the gradient free energy density accounting for the effect of the interfaces is defined as  
\begin{equation}\label{eqHelmholtzfreeenergydensity02}
 f_\grad(c)=\frac{1}{2}\kappa|\grad c|^2, 
\end{equation}
where $\kappa>0$ is the influence parameter that can be calculated as follows
\begin{eqnarray*}
 \kappa=\alpha\beta^{2/3}\[a_0(1-T_{r})+a_1\],
\end{eqnarray*}
\begin{eqnarray*}
 a_0=-\frac{10^{-16}}{1.2326+1.3757\omega},~~~~a_1=\frac{10^{-16}}{0.9051+1.5410\omega}.
\end{eqnarray*}

We denote by $f(c)$  the general Helmholtz free energy density   
\begin{eqnarray}\label{eqgeneralchemicalpotential}
    f(c)=f_b(c)+f_\grad(c).
\end{eqnarray}
The chemical potential   is defined  as the variational derivative of $f(c)$ 
\begin{eqnarray}\label{eqgeneralchemicalpotential}
    \mu(c)=\frac{\delta f(c)}{\delta c}=\mu_b(c)-\kappa \Delta c,
\end{eqnarray}
 where $\mu_b(c)=f'_b(c)$ is the bulk chemical potential.

Let $\Omega$ be a connected and smooth space domain.  We now state the model equation as follows \cite{qiaosun2014}
 \begin{subequations}\label{eqOriginalEqns}
\begin{equation}\label{eqOriginalEqns01}
\frac{\partial c}{\partial t}-\kappa \Delta c +\mu_b(c)=\mu_e,
\end{equation}
\begin{equation}\label{eqOriginalEqns02}
\int_\Omega c d\x=c_t,
\end{equation}
\begin{equation}
\grad c \cdot\bm{n}_{\partial\Omega}=0,~~\x\in \partial\Omega,
~~~c(\x,0)=c_0(\x),~~\x\in\Omega,
\end{equation}
\end{subequations}
 where $\bm{n}_{\partial\Omega}$ denotes the normal unit outward vector  to  $\partial \Omega$,  $c_t>0$ is the total moles in $\Omega$  and $\mu_e$ is a Lagrange multiplier  incorporated to enforce total moles conservation.  We note that $\mu_e$  is constant in space but could vary with time. Moreover, as   time   goes on,  the spatial distribution of molar density    $c$ will approach  to an equilibrium  state, and $\mu_e$ will also approach the chemical potential at the equilibrium state.
Here, we consider the homogeneous Neumann boundary condition, but the proposed numerical schemes and theoretical analysis can be directly extended to various boundary conditions, for instance,  Dirichlet boundary conditions.

We note that the model \eqref{eqOriginalEqns} obeys the  energy dissipation law  
\begin{equation}\label{eqOriginalEqns01EnergyDiss}
\frac{\partial }{\partial t}\int_\Omega f(c(\x,t)) d\x=-\int_\Omega\(\frac{\partial c}{\partial t}\)^2d\x.
\end{equation}

\section{Energy factorization approach}\label{secEF}
In this section, we propose a novel energy factorization (EF) approach to construct  a linear, energy stable time matching scheme for the model \eqref{eqOriginalEqns}. 
  The basic idea of EF  is to factorize an energy function or a term of the energy function into   a product of several factors that can be separately treated by the use of their properties. Two different  factorizations are proposed to deal with the ideal term and the repulsion term of the bulk Helmholtz free energy density respectively.

At the time discrete level, we denote the time step size   by $\tau$ and set $t_n=n\tau$. Furthermore, we use $c^{n}$ to denote the approximation of molar  density $c$ at the time $t_n$. 

\subsection{Factorization approach for  the ideal term}
The first energy factorization approach is proposed to deal with  the ideal term. We define the function $H(c)=c\ln(c)$, which  can be factorized into the product of a linear function $c$  and a  logarithm function $\ln(c)$. Apparently, $\ln(c)$ is a concave function, and thus, we have
\begin{equation}\label{eqConcavityoflogarithm}
\ln(c^{n+1})\leq\ln(c^n)+\frac{1}{c^{n}}\(c^{n+1}-c^{n}\).
\end{equation}
 For $c^n>0$ and $c^{n+1}>0$, 
 using \eqref{eqConcavityoflogarithm}, we can deduce that 
 \begin{eqnarray}\label{eqIdealTermFactorizatopn01}
H(c^{n+1})-H(c^n)&=&c^{n+1}\ln\(c^{n+1})-c^{n}\ln(c^{n}\)\nonumber\\
&=&\ln(c^{n})\(c^{n+1}-c^{n}\)+c^{n+1}\(\ln(c^{n+1})-\ln(c^{n})\)\nonumber\\
&\leq&\(\ln\(c^{n}\)+\frac{c^{n+1}}{c^{n}}\)\(c^{n+1}-c^{n}\).
\end{eqnarray}
Then the   ideal term $f_b^{\textnormal{ideal}}(c)$ can be estimated as
\begin{equation}\label{eqIdealTermFactorizatopn02}
f_b^{\textnormal{ideal}}(c^{n+1})-f_b^{\textnormal{ideal}}(c^n)\leq\vartheta_0\(c^{n+1}-c^{n}\)+RT\(\ln\(c^{n}\)+\frac{c^{n+1}}{c^{n}}\)\(c^{n+1}-c^{n}\),
\end{equation}
which suggests us to define the ideal part of chemical potential at the $(n+1)$-th time step as
\begin{eqnarray}\label{eqIdealTermFactorizatopn03}
\mu_{\textnormal{ideal}}^{n+1}=\vartheta_0+RT\ln\(c^{n}\)+RT\frac{c^{n+1}}{c^{n}}.
\end{eqnarray}
Therefore, from \eqref{eqIdealTermFactorizatopn02}, we have
\begin{align}\label{eqEnergyFactorizationIdeal}
f_b^{\textnormal{ideal}}(c^{n+1})-f_b^{\textnormal{ideal}}(c^n)\leq \mu_{{\textnormal{ideal}}}^{n+1}\(c^{n+1}-c^n\).
\end{align}
 
We remark that the convex splitting approach   for the ideal term $f_b^{\textnormal{ideal}}(c)$ used in \cite{qiaosun2014} leads to a highly nonlinear scheme,  which demands  the complicated  implementations of  efficient nonlinear iterative solvers and     also costs  expensively  in computations. In contrast, $\mu_{\textnormal{ideal}}^{n+1}$ is semi-implicit and linear with respect to $c^{n+1}$,  thus it is   easy-to-implement and can be solved at   less computational costs.
This approach can also be directly applied for the logarithmic Flory-Huggins potential \cite{Yang2019IEQ,Zhu2018SAV}.

\subsection{Factorization approach for the repulsion term}
Since the function $-\ln\(1-\beta c\)$ is a convex function, we would obtain a  nonlinear chemical potential when the first energy factorization approach dealing with the ideal term is employed for  the repulsion term $f_b^{\textnormal{repulsion}}(c)$. In order to pursue a linear scheme,  we introduce the second energy factorization approach to deal with the repulsion term $f_b^{\textnormal{repulsion}}(c)$. 

 We      define the modified  repulsion energy function as
 \begin{equation}\label{eqEnergyFactorizationRepulsion01}
\widehat{f}_b^{\textnormal{repulsion}}(c)=\lambda c+\frac{1}{RT}f_b^{\textnormal{repulsion}}(c)=\lambda c-c\ln\(1-\beta c\),
\end{equation}
where $\lambda $ is some positive constant relying on the specific   substance.
Apparently,  $\widehat{f}_b^{\textnormal{repulsion}}(c)$ is positive and bounded   for molar density $c$ satisfying \eqref{eqMolDensRange}.
We further define the following intermediate energy function 
\begin{equation}\label{eqEnergyFactorizationRepulsion02}
G(c)=\sqrt{\lambda c-c\ln\(1-\beta c\)}.
\end{equation}
From this, we can  factorize $\widehat{f}_b^{\textnormal{repulsion}}(c)$   into the square of $G(c)$, i.e.,   
\begin{equation}\label{eqEnergyFactorizationRepulsion02A}
\widehat{f}_b^{\textnormal{repulsion}}(c)=G(c)^2.
\end{equation}

For the function $G(c)$, we have the following key lemma regarding the choice of $\lambda$.
\begin{lem}\label{lemEnergyFactorizationRepulsion01}
Assume that molar density $c$ satisfies \eqref{eqMolDensRange}. If $\lambda$ is taken such that
\begin{equation}\label{eqEnergyFactorizationRepulsionConcaveCondition}
\lambda\geq     \frac{ \epsilon_0}{(1-\epsilon_0)^2}+\(   \frac{\epsilon_0^2}{(1-\epsilon_0)^4}-2 \ln\(1-\epsilon_0\)\frac{\epsilon_0}{(1-\epsilon_0)^2}\)^{1/2},
\end{equation}
where $\epsilon_0$ is given in \eqref{eqMolDensRange}, then $G(c)$ is a concave function.
\end{lem}
\begin{proof}
We calculate the first and second derivatives of $G(c)$ as follows
\begin{eqnarray}\label{eqEnergyFactorizationRepulsionConcaveConditionProof01}
G'(c)=\frac{1}{2}G(c)^{-1}\(\lambda-\ln(1-\beta c)+\frac{\beta c}{1-\beta c}\),
\end{eqnarray}
\begin{align}\label{eqEnergyFactorizationRepulsionConcaveConditionProof02}
G''(c)
&=-\frac{1}{4}G(c)^{-3}\(\lambda- \ln(1-\beta c)+ \frac{\beta c}{1-\beta c}\)^2\nonumber\\
&~~~~+\frac{1}{2}G(c)^{-1} \(\frac{\beta}{1-\beta c}+\frac{\beta}{(1-\beta c)^2}\)\nonumber\\
&=-\frac{1}{4}G(c)^{-3}\left(\lambda^2- 2\lambda \ln(1-\beta c)+\(\ln(1-\beta c)\)^2+ \(\frac{\beta c}{1-\beta c}\)^2\right.\nonumber\\
&~~~~\left.-2 \lambda \frac{\beta c}{(1-\beta c)^2}+2 \ln\(1-\beta c\)\frac{\beta c}{(1-\beta c)^2}\right)\nonumber\\
&\leq-\frac{1}{4}G(c)^{-3}\left(\lambda^2-2 \lambda \frac{\beta c}{(1-\beta c)^2}+2 \ln\(1-\beta c\)\frac{\beta c}{(1-\beta c)^2}\right)\nonumber\\
&\leq-\frac{1}{4}G(c)^{-3}\left(\lambda^2-2 \lambda \frac{\epsilon_0}{(1-\epsilon_0)^2}+2 \ln\(1-\epsilon_0\)\frac{\epsilon_0}{(1-\epsilon_0)^2}\right).
\end{align}
Applying the condition \eqref{eqEnergyFactorizationRepulsionConcaveCondition}, we obtain $G''(c)\leq0$, thus $G(c)$ is   concave.
\end{proof}

We can see from the proof of Lemma \ref{lemEnergyFactorizationRepulsion01} that  the condition \eqref{eqEnergyFactorizationRepulsionConcaveCondition}    can be relaxed further substantially. Moreover, $\lambda$ usually does not demand  a large value  in practice since $\epsilon_0$  generally is not very small value; for instance,  we calculate  from \eqref{eqEnergyFactorizationRepulsionConcaveCondition} that $\lambda=27.3656$ in numerical examples. We also note   that $\lambda$ is dimensionless.

The advantage of    the factorization \eqref{eqEnergyFactorizationRepulsion02A} is shown in the following lemma.
\begin{lem}\label{lemEnergyFactorizationRepulsionG} 
Assume that   molar density   satisfies \eqref{eqMolDensRange}  and $\lambda$ is taken such that
\eqref{eqEnergyFactorizationRepulsionConcaveCondition} holds. Then
we have
\begin{align}\label{eqEnergyFactorizationRepulsionG}
G(c^{n+1})^2-G(c^n)^2\leq \(G^{n+1}+G(c^n)\)G'(c^n)\(c^{n+1}-c^n\),
\end{align}
where $G^{n+1}$ is  the   linear approximation of $G(c^{n+1})$ as   
\begin{equation}\label{eqEnergyFactorizationRepulsion05}
G^{n+1}=G(c^n)+G'(c^n)\(c^{n+1}-c^n\).
\end{equation}
\end{lem}
\begin{proof}
The assumption implies  the concavity of $G(c)$, so we have
\begin{align}\label{eqEnergyFactorizationRepulsionProof01G}
G(c^{n+1})\leq G(c^n)+G'(\phi^n)\(c^{n+1}-c^n\)=G^{n+1}.
\end{align}
Since $G(c)\geq0$, we get $G^{n+1}\geq0$ from \eqref{eqEnergyFactorizationRepulsionProof01G} and consequently 
\begin{equation}\label{eqEnergyFactorizationRepulsionProof02G}
G(c^{n+1})^2\leq |G^{n+1}|^2.
\end{equation}
From \eqref{eqEnergyFactorizationRepulsion05} and \eqref{eqEnergyFactorizationRepulsionProof02G}, we derive   that
\begin{align}\label{eqEnergyFactorizationRepulsionProof03G}
G(c^{n+1})^2-G(c^n)^2&\leq |G^{n+1}|^2-G(c^n)^2\nonumber\\
&=\(G^{n+1}+G(c^n)\)\(G^{n+1}-G(c^n)\)\nonumber\\
&=\(G^{n+1}+G(c^n)\)G'(c^n)\(c^{n+1}-c^n\).
\end{align}
This ends the proof.
\end{proof}

 We are now ready to consider treatment of $f_b^{\textnormal{repulsion}}(c)$ under the condition \eqref{eqEnergyFactorizationRepulsionConcaveCondition}. 
Using Lemma \ref{lemEnergyFactorizationRepulsionG},  we can estimate the energy difference between  two time steps   as
\begin{align}\label{eqEnergyFactorizationRepulsion04}
\widehat{f}_b^{\textnormal{repulsion}}(c^{n+1})-\widehat{f}_b^{\textnormal{repulsion}}(c^{n})&= G(c^{n+1})^2-G(c^{n})^2\nonumber\\
&\leq  \(G^{n+1}+G(c^n)\)G'(c^n)\(c^{n+1}-c^n\).
\end{align}
It follows from \eqref{eqEnergyFactorizationRepulsion01} and \eqref{eqEnergyFactorizationRepulsion04} that
\begin{align}\label{eqEnergyFactorizationRepulsionProof05}
f_b^{\textnormal{repulsion}}(c^{n+1})-f_b^{\textnormal{repulsion}}(c^{n})
&=RT\big(\widehat{f}_b^{\textnormal{repulsion}}(c^{n+1})-\widehat{f}_b^{\textnormal{repulsion}}(c^{n})\big)\nonumber\\
&~~~~-RT\lambda\(c^{n+1}-c^n\)\nonumber\\
&\leq RT\(\(G^{n+1}+G(c^n)\)G'(c^n)-\lambda\)\(c^{n+1}-c^n\).
\end{align}
Thus,
we can define the repulsion chemical potential      at the $(n+1)$-th time step as
\begin{equation}\label{eqEnergyFactorizationRepulsionChemicalPotential}
\mu_{{\textnormal{repulsion}}}^{n+1}=RTG'(c^n)\(G^{n+1}+G(c^n)\)-\lambda RT.
\end{equation}
Substituting \eqref{eqEnergyFactorizationRepulsion05} into \eqref{eqEnergyFactorizationRepulsionChemicalPotential}, we rewrite  
\begin{equation}\label{eqEnergyFactorizationRepulsionChemicalPotentialA}
\mu_{{\textnormal{repulsion}}}^{n+1}
=RTG'(c^n)\(2G(c^n)+G'(c^n)\(c^{n+1}-c^n\)\)-\lambda RT,
\end{equation}
which is a linear function of $c^{n+1}$.  We note that the convex splitting approach   for the repulsion term $f_b^{\textnormal{repulsion}}(c)$   \cite{qiaosun2014} results in a   nonlinear scheme as well as the ideal term. 

The following lemma is a direct consequence  of the above analysis and the definition of $\mu_{{\textnormal{repulsion}}}^{n+1}$. 
\begin{lem}\label{lemEnergyFactorizationRepulsion} 
Assume that   molar density   satisfies \eqref{eqMolDensRange} and $\lambda$ is taken such that
\eqref{eqEnergyFactorizationRepulsionConcaveCondition} holds. Then
we have
\begin{align}\label{eqEnergyFactorizationRepulsion}
f_b^{\textnormal{repulsion}}(c^{n+1})-f_b^{\textnormal{repulsion}}(c^{n})\leq \mu_{{\textnormal{repulsion}}}^{n+1}\(c^{n+1}-c^n\).
\end{align}
\end{lem}

We  remark  that the proposed approach is different from the IEQ and SAV approaches. In the IEQ and SAV approaches,   some new auxiliary energy variables are introduced,  and from this,   the original    energy is transformed to a quadratic form.  Nevertheless, the transformed    energy  is generally  not equivalent to the original   energy at the time discrete level  \cite{Yang2019IEQ}.  In the proposed approach,  we just use  $G(c)$ as a function of $c$, but never introducing  any new independent variable. This is a key feature of the proposed approach that allows  us to apply the concavity of $G(c)$ to obtain the linear  numerical scheme inheriting  the dissipation law of the original   energy at the discrete level.

\section{Semi-implicit time-discrete scheme}
In this section, we propose a semi-implicit time-discrete scheme based on the results presented in Section \ref{secEF}.
The ideal and repulsion terms have been handled by the EF approach. Due to the concavity of the attraction term \cite{qiaosun2014},  we     treat it explicitly   and define the corresponding  chemical potential term as
\begin{equation}
\mu_{{\textnormal{attraction}}}(c^{n})=\frac{\alpha(T)}{2\sqrt{2}\beta}\ln\(\frac{1+(1-\sqrt{2})\beta c^n}{1+(1+\sqrt{2})\beta c^n}\)
-\frac{\alpha(T)c^n}{1+2\beta c^n-(\beta  c^n)^2}.
\end{equation}
Let $\tau$ be the time step size and let $c^0$ be provided by the initial condition, we now state the semi-implicit linear time-discrete scheme as follows
\begin{subequations}\label{eqSemiDiscreteOriginalEqns}
\begin{equation}\label{eqSemiDiscreteOriginalEqns01}
\frac{c^{n+1}-c^n}{\tau }-\kappa \Delta c^{n+1}+\mu_{{\textnormal{ideal}}}^{n+1}+\mu_{{\textnormal{repulsion}}}^{n+1}+\mu_{{\textnormal{attraction}}}(c^{n})=\mu_e^{n+1},
\end{equation}
\begin{equation}\label{eqSemiDiscreteOriginalEqns02}
\int_\Omega c^{n+1}d\x=c_t,
\end{equation}
\begin{equation}
\grad c^{n+1}\cdot\bm{n}_{\partial\Omega}=0,~~\x\in \partial\Omega,
\end{equation}
\end{subequations}
where   $\mu_{{\textnormal{ideal}}}^{n+1}$ and $\mu_{{\textnormal{repulsion}}}^{n+1}$ are defined in 
\eqref{eqIdealTermFactorizatopn03} and \eqref{eqEnergyFactorizationRepulsionChemicalPotentialA} respectively.
For   the convenience of  theoretical analysis,  we rewrite \eqref{eqSemiDiscreteOriginalEqns} as the following equivalent form 
\begin{subequations}\label{eqSemiDiscreteEqns}
\begin{equation}\label{eqSemiDiscreteEqnsA}
\frac{1}{\tau}\(c^{n+1}-c^n\)-\kappa \Delta c^{n+1}+\nu(c^n)c^{n+1}=s_r(c^n)+\mu_e^{n+1},
\end{equation}
\begin{equation}\label{eqSemiDiscreteEqnsB}
\int_\Omega c^{n+1}d\x=c_t,
\end{equation}
\begin{equation}\label{eqSemiDiscreteEqnsC}
\grad c^{n+1}\cdot\bm{n}_{\partial\Omega} =0 ~~\textnormal{on} ~ \partial\Omega,
\end{equation}
\end{subequations}
where  $\nu(c)$ and $s_r(c)$ are the functions of $c$ defined as follows
\begin{equation}\label{eqDiscreteEqnsLinTerm}
\nu(c)=RT\(\frac{ 1}{c}+G'(c)^2\),
\end{equation}
\begin{equation}\label{eqDiscreteEqnsSourceTerm}
s_r(c)=-\vartheta_0-RT\ln\(c\)+RT\(G'(c)^2c-2G(c)G'(c)+\lambda\)
-\mu_{{\textnormal{attraction}}}(c).
\end{equation}

In what follows, we use the traditional notations to denote the inner product of $L^2(\Omega)$ and $\(L^2(\Omega)\)^d$  by $(\cdot,\cdot)$ and the norm of $L^2(\Omega)$ and $\(L^2(\Omega)\)^d$ by $\|\cdot\|$, where $d$ is the spatial dimension.
     
\subsection{ Well-posedness of the solution}\label{secSemiDiscreteScheme}

We now show the existence and uniqueness of the solution  of the semi-discrete scheme \eqref{eqSemiDiscreteEqns} as follows. 
\begin{thm}\label{thmExistenceOfMinimizerOfF}
Assume  that $c^n$ satisfies the condition \eqref{eqMolDensRange} and $\lambda$ is taken to satisfy 
\eqref{eqEnergyFactorizationRepulsionConcaveCondition}. There exists a unique      $c^{n+1}$ to solve
\eqref{eqSemiDiscreteEqns} weakly in $H^1(\Omega)$. 
\end{thm}

\begin{proof}
Suppose that $c$ is a solution of the following homogeneous problem
\begin{subequations}\label{eqSemiDiscreteEqnsProof02}
\begin{equation}\label{eqSemiDiscreteEqnsProof02a}
\frac{1}{\tau}c-\kappa \Delta c+\nu(c^n)c=\mu_e,
\end{equation}
\begin{equation}\label{eqSemiDiscreteEqnsProof02b}
\int_\Omega c d\x=0, 
\end{equation}
\begin{equation}\label{eqSemiDiscreteEqnsProof02c}
 \grad c\cdot\bm{n}_{\partial\Omega}  =0 ~~\textnormal{on} ~ \partial\Omega.
\end{equation}
\end{subequations}
By Fredholm alternative theorem, it suffices to  prove $c\equiv0$.
Multiplying the   equation \eqref{eqSemiDiscreteEqnsProof02a} by $c$ and integrating it over $\Omega$, we obtain 
\begin{equation}\label{eqSemiDiscreteEqnsUniqueProof02}
\frac{1}{\tau}\|c\|^2+\kappa \|\grad c\|^2+\(\nu(c^n),c^2\)=\(\mu_e,c\).
\end{equation}
The right-hand side term vanishes due to  \eqref{eqSemiDiscreteEqnsProof02b}.
Furthermore, $\nu(c^n)>0$ holds for $c_m\leq c^n\leq c_M$.
The equation \eqref{eqSemiDiscreteEqnsUniqueProof02} can be reduced into
\begin{equation}
\frac{1}{\tau}\|c\|^2+\kappa \|\grad c\|^2\leq0.
\end{equation}
This means that  $c\equiv0$ almost everywhere in $\Omega$  as well as $\mu_e=0$   from \eqref{eqSemiDiscreteEqnsProof02a}. 
\end{proof}

\subsection{Maximum principle}
We now prove that the time scheme given in \eqref{eqSemiDiscreteOriginalEqns} follows the maximum principle, which can ensure  that the schemes presented in Section \ref{secEF} are always well defined.
\begin{thm}\label{thmMaximumPrinciple}
Assume  that $c^n$ satisfies the condition \eqref{eqMolDensRange} and $\lambda$ is taken to satisfy 
\eqref{eqEnergyFactorizationRepulsionConcaveCondition}.
If 
\begin{equation}\label{eqMaximumPrincipleCondition}
 \max_{c_m\leq c\leq c_M}\(c_m\nu(c)-s_r(c)\)\leq\mu_e^{n+1}\leq \min_{c_m\leq c\leq c_M}\(c_M\nu(c)-s_r(c)\),
\end{equation}
then we have $c_m\leq c^{n+1}\leq c_M$ almost everywhere.
\end{thm}
 
\begin{proof}
Suppose that $c_m\leq c^n\leq c_M$ holds for $n\geq0$. We first prove that $c^{n+1}\geq c_m$  almost everywhere.  Let $c_{-}^{n+1}=\min(c^{n+1}-c_m,0)$ and apparently  $c_{-}^{n+1}\leq0$. Multiplying the   equation  \eqref{eqSemiDiscreteEqnsA}   by 
$c_{-}^{n+1}$ and then integrating it over $\Omega$, we have
\begin{align}\label{eqSemiDiscreteSchmMaximumProof01}
\frac{1}{\tau}\(c^{n+1}-c^n,c_{-}^{n+1}\)-\kappa \(\Delta c^{n+1},c_{-}^{n+1}\)
+\(\nu(c^n)(c^{n+1}-c_m),c_-^{n+1}\)\nonumber\\
=\(\mu_e^{n+1},c_-^{n+1}\)+\(s_r(c^n)-c_m\nu(c^n),c_-^{n+1}\).
\end{align}
For the first term on the left-hand side of \eqref{eqSemiDiscreteSchmMaximumProof01}, thanks to $c^n\geq c_m$, we deduce
\begin{align}\label{eqSemiDiscreteSchmMaximumProof02}
\(c^{n+1}-c^n,c_{-}^{n+1}\)=\|c_{-}^{n+1}\|^2-\(c^n-c_m,c_{-}^{n+1}\)\geq\|c_{-}^{n+1}\|^2.
\end{align}
By using the boundary condition, the second term on the left-hand side of \eqref{eqSemiDiscreteSchmMaximumProof01} becomes 
\begin{align}\label{eqSemiDiscreteSchmMaximumProof03}
-\kappa \(\Delta c^{n+1},c_{-}^{n+1}\)=\kappa\|\grad c_{-}^{n+1}\|^2.
\end{align}
We observe that $\nu(c)$ is a strictly monotonically decreasing positive function  over the interval $[c_m,c_M]$. The   third term  on the left-hand side of \eqref{eqSemiDiscreteSchmMaximumProof01} is bounded below  
\begin{align}\label{eqSemiDiscreteSchmMaximumProof04}
\(\nu(c^n)(c^{n+1}-c_m),c_-^{n+1}\)=\(\nu(c^n),|c_-^{n+1}|^2\)\geq\nu(c_M)\|c_-^{n+1}\|^2.
\end{align}
Using the condition \eqref{eqMaximumPrincipleCondition} and taking into account $c_{-}^{n+1}\leq0$, we estimate the first term on the right-hand side of \eqref{eqSemiDiscreteSchmMaximumProof01}   as
\begin{align}\label{eqSemiDiscreteSchmMaximumProof05}
\(\mu_e^{n+1},c_-^{n+1}\)
&\leq\(\max_{c_m\leq c\leq c_M}\(c_m\nu(c)-s_r(c)\),c_-^{n+1}\) \nonumber\\
&\leq\(c_m\nu(c^n)-s_r(c^n),c_-^{n+1}\).
\end{align}
Combining \eqref{eqSemiDiscreteSchmMaximumProof02}-\eqref{eqSemiDiscreteSchmMaximumProof05}, we derive from \eqref{eqSemiDiscreteSchmMaximumProof01} that
\begin{align}\label{eqSemiDiscreteSchmMaximumProof06}
\frac{1}{\tau}\|c_{-}^{n+1}\|^2+\kappa\|\grad c_{-}^{n+1}\|^2+\nu(c_M)\|c_-^{n+1}\|^2
\leq0.
\end{align}
It follows from \eqref{eqSemiDiscreteSchmMaximumProof06} that
 $\|c_{-}^{n+1}\|^2=0$ and $\|\grad c_{-}^{n+1}\|^2=0$. Consequently, $c^{n+1}\geq c_m$  almost everywhere.

We turn to prove $c^{n+1}\leq c_M$   almost everywhere. We define $c_{+}^{n+1}=\max(c^{n+1}-c_M,0)$ and apparently  $c_{+}^{n+1}\geq0$. 
Similar to \eqref{eqSemiDiscreteSchmMaximumProof01}, we can get
\begin{align}\label{eqSemiDiscreteSchmMinimumProof01}
&\frac{1}{\tau}\(c^{n+1}-c^n,c_{+}^{n+1}\)-\kappa \(\Delta c^{n+1},c_{+}^{n+1}\)
+\(\nu(c^n)(c^{n+1}-c_M),c_+^{n+1}\)\nonumber\\
&=\(\mu_e^{n+1},c_+^{n+1}\)+\(s_r(c^n)-\nu(c^n)c_M,c_+^{n+1}\).
\end{align}
Taking into account $c^n\leq c_M$, we deduce
\begin{align}\label{eqSemiDiscreteSchmMinimumProof02}
\(c^{n+1}-c^n,c_{+}^{n+1}\)=\|c_{+}^{n+1}\|^2-\(c^n-c_M,c_{+}^{n+1}\)\geq\|c_{+}^{n+1}\|^2.
\end{align}
Using the condition \eqref{eqMaximumPrincipleCondition} and   $c_{+}^{n+1}\geq0$, we derive
\begin{align}\label{eqSemiDiscreteSchmMinimumProof03}
\(\mu_e^{n+1},c_-^{n+1}\)
&\leq\(\min_{c_m\leq c\leq c_M}\(c_M\nu(c)-s_r(c)\),c_+^{n+1}\) \nonumber\\
&\leq\(c_M\nu(c^n)-s_r(c^n),c_+^{n+1}\).
\end{align}Using the similar routines as in the derivations of  \eqref{eqSemiDiscreteSchmMaximumProof06}, we can obtain
\begin{align}\label{eqSemiDiscreteSchmMinimumProof04}
\frac{1}{\tau}\|c_{+}^{n+1}\|^2+\kappa\|\grad c_{+}^{n+1}\|^2+\nu(c_M)\|c_+^{n+1}\|^2
\leq0.
\end{align}
This implies that
 $\|c_{+}^{n+1}\|^2=0$ and $\|\grad c_{+}^{n+1}\|^2=0$, and thus,  $c^{n+1}\leq c_M$   almost everywhere.
\end{proof}

We remark that it is reasonable to assume the condition \eqref{eqMaximumPrincipleCondition}, as it is likely to be required in a typical physical setting. As a matter of fact, we can    derive its a   priori   bounds just  assuming  $c^{n+1}>0$. 
Integrating the equation \eqref{eqSemiDiscreteEqnsA} over $\Omega$ and using the constraint \eqref{eqSemiDiscreteEqnsB}, we get 
\begin{equation}
\mu_e^{n+1}=\frac{1}{|\Omega|}\int_\Omega\(\nu(c^n)c^{n+1}- s_r(c^n)\)d\x,
\end{equation}
where $|\Omega|$ is the measure of $\Omega$. Using the constraint \eqref{eqSemiDiscreteEqnsB} agian, we can estimate 
\begin{align}
\int_\Omega\nu(c^n)c^{n+1} d\x\leq\max\(\nu(c^n)\)\int_\Omega c^{n+1} d\x \leq \nu(c_m)c_t,
\end{align}
\begin{align}
\int_\Omega\nu(c^n)c^{n+1} d\x\geq \min\(\nu(c^n)\)\int_\Omega c^{n+1} d\x\geq\nu(c_M)c_t.
\end{align}
Let us denote $s_{r}^m=\min_{c\in[c_m,c_M]}(s_r(c))$ and $s_{r}^M=\max_{c\in[c_m,c_M]}(s_r(c))$. Thus, $\mu_e^{n+1}$ is bounded as
\begin{equation}\label{eqMaximumPrincipleConditionRelaxed}
\nu(c_M)\frac{c_t}{|\Omega|}- s_{r}^M\leq\mu_e^{n+1}\leq \nu(c_m)\frac{c_t}{|\Omega|}- s_{r}^m.
\end{equation}
 Numerical results will also be presented to verify  \eqref{eqMaximumPrincipleCondition}.

\subsection{Energy stability}
We define the total free energy as
\begin{equation}\label{eqSemiSchmEnergy}
F(c^n)=\(f_b(c^{n}),1\)+\frac{1}{2}\kappa\|\grad c^{n}\|^2.
\end{equation}

The following theorem shows that the proposed scheme inherits  the dissipation law of the original   energy at the discrete level.

\begin{thm}\label{thmSemiSchmEnergyStability}
Assume  that $c^n$ satisfies   \eqref{eqMolDensRange} and $\lambda$ is taken to satisfy 
\eqref{eqEnergyFactorizationRepulsionConcaveCondition}. 
Under the condition \eqref{eqMaximumPrincipleCondition}, for any time step size $\tau$, we have  
\begin{equation}\label{eqSemiSchmEnergyStability}
F(c^{n+1})\leq F(c^n).
\end{equation}
\end{thm}
\begin{proof}
It follows from  the concavity of $ f_b^{\textnormal{attraction}}(c)$ that
 \begin{align}\label{eqSemiSchmEnergyStabilityProof00}
\(f_b^{\textnormal{attraction}}(c^{n+1})-f_b^{\textnormal{attraction}}(c^{n}),1\)\leq \(\mu_{{\textnormal{attraction}}}(c^{n}),c^{n+1}-c^n\).
\end{align}
Applying \eqref{eqEnergyFactorizationIdeal}, \eqref{eqEnergyFactorizationRepulsion} and \eqref{eqSemiSchmEnergyStabilityProof00}, we deduce that
\begin{align}\label{eqSemiSchmEnergyStabilityProof01}
\(f_b(c^{n+1})-f_b(c^{n}),1\)\leq \(\mu_{{\textnormal{ideal}}}^{n+1}+\mu_{{\textnormal{repulsion}}}^{n+1}+\mu_{{\textnormal{attraction}}}(c^{n}),c^{n+1}-c^n\).
\end{align}
For the gradient contribution to the free energy, we can derive that
\begin{align}\label{eqSemiSchmEnergyStabilityProof02}
\frac{1}{2}\(\|\grad c^{n+1}\|^2-\|\grad c^{n}\|^2\)
&=\frac{1}{2}\int_\Omega\(|\grad c^{n+1}|^2-|\grad c^{n}|^2\)d\x\nonumber\\
&=\int_\Omega \grad c^{n+1}\cdot\grad\(c^{n+1}-c^{n}\)d\x
-\frac{1}{2}\int_\Omega|\grad(c^{n+1}-c^n)|^2d\x\nonumber\\
&\leq\int_\Omega \grad c^{n+1}\cdot\grad\(c^{n+1}-c^{n}\)d\x\nonumber\\
&=-\int_\Omega \(c^{n+1}-c^{n}\)\Delta c^{n+1} d\x.
\end{align}
Taking into account the equation \eqref{eqSemiDiscreteOriginalEqns01} and the mass constraint \eqref{eqSemiDiscreteOriginalEqns02}, we deduce from \eqref{eqSemiSchmEnergyStabilityProof01} and \eqref{eqSemiSchmEnergyStabilityProof02} that
\begin{align}\label{eqSemiSchmEnergyStabilityProof03}
F(c^{n+1})-F(c^n)&=\(f_b(c^{n+1})-f_b(c^{n}),1\)+\frac{1}{2}\kappa\(\|\grad c^{n+1}\|^2-\|\grad c^{n}\|^2\)\nonumber\\
&\leq\(\mu_{{\textnormal{ideal}}}^{n+1}+\mu_{{\textnormal{repulsion}}}^{n+1}+\mu_{{\textnormal{attraction}}}(c^{n})-\kappa\Delta c^{n+1}, c^{n+1}-c^{n}\)\nonumber\\
&=\(\mu_e^{n+1}-\frac{c^{n+1}-c^n}{\tau}, c^{n+1}-c^{n}\)\nonumber\\
&=-\frac{1}{\tau}\|c^{n+1}-c^{n}\|^2.
\end{align}
Thus, the energy inequality \eqref{eqSemiSchmEnergyStability} is proved.
\end{proof}
\section{Fully discrete scheme}
In this section, we consider the fully discrete  scheme.  The cell-centered finite difference (CCFD) method \cite{Tryggvason2011book} is employed as the spatial discretization method. We note that the CCFD method is equivalent to a mixed finite element method  with quadrature rules \cite{arbogast1997mixed}. 
Here, we present  the numerical scheme in two-dimensional case only, but it is straightforward to extend it to the  three-dimensional case.

We consider a rectangular domain  as $\Omega=[0,l_x]\times[0,l_y]$, where $l_x>0$ and $l_y>0$.   For simplicity,   a uniform mesh of  $\Omega$  is used as 
$0=x_0<x_1<\cdots<x_{N}=l_x$ and $0=y_0<y_1<\cdots<y_{M}=l_y,$
where $N$ and $M$ are   integers. We also introduce the intermediate points  $x_{i+\frac{1}{2}} =\frac{x_i+x_{i+1}}{2}$ and $y_{j+\frac{1}{2}} =\frac{y_j+y_{j+1}}{2}$. The mesh size is denoted by $h=x_{i+1}-x_i=y_{j+1}-y_j$.

\subsection{Notations and fully discrete scheme}

To formulate the fully discrete scheme, we define the following discrete function spaces:
   \begin{equation*}
\mathcal{V}_c=\big\{c: (x_{i+\frac{1}{2}},y_{j+\frac{1}{2}})\mapsto\mathbb{R},~~0\leq i\leq N-1, ~0\leq j\leq M-1\big\},
  \end{equation*}
\begin{equation*}
\mathcal{V}_u=\big\{u: (x_{i},y_{j+\frac{1}{2}})\mapsto\mathbb{R},~~0\leq i\leq N, ~0\leq j\leq M-1\big\},
  \end{equation*}
\begin{equation*}
\mathcal{V}_v=\big\{v: (x_{i+\frac{1}{2}},y_{j})\mapsto\mathbb{R},~~0\leq i\leq N-1, ~0\leq j\leq M\big\}.
  \end{equation*}
For components of discrete functions  in the above spaces, we  denote  
 $c_{i+\frac{1}{2},j+\frac{1}{2}}=c(x_{i+\frac{1}{2}},y_{j+\frac{1}{2}})$ for $c\in \mathcal{V}_c$, $u_{i,j+\frac{1}{2}}=u(x_{i},y_{j+\frac{1}{2}})$ for $u\in \mathcal{V}_u$ and $v_{i+\frac{1}{2},j}=v(x_{i+\frac{1}{2}},y_{j})$ for $v\in \mathcal{V}_v$.

For $c\in \mathcal{V}_c$, we define the   difference operators  $\delta_x^c[c]\in\mathcal{V}_u$ and $\delta_y^c[c]\in\mathcal{V}_v$ as follows
\begin{subequations}\label{eqFullyDiscreteGradofMolarDens}
\begin{equation}\label{eqFullyDiscreteGradofMolarDens01}
\delta_x^c[c]_{i,j+\frac{1}{2}}=\frac{c_{i+\frac{1}{2},j+\frac{1}{2}}-c_{i-\frac{1}{2},j+\frac{1}{2}}}{h},~~1\leq i\leq N-1, ~0\leq j\leq M-1,
  \end{equation}
\begin{equation}\label{eqFullyDiscreteGradofMolarDens02}
\delta_y^c[c]_{i+\frac{1}{2},j}= 
\frac{c_{i+\frac{1}{2},j+\frac{1}{2}}-c_{i+\frac{1}{2},j-\frac{1}{2}}}{h},~~0\leq i\leq N-1, ~1\leq j\leq M-1.
  \end{equation}
 \end{subequations}
 On the boundary, 
applying the homogeneous  Neumann boundary condition, we take  the difference operators    as  
 \begin{subequations}\label{eqFullyDiscreteGradofMolarDens}
\begin{equation}\label{eqFullyDiscreteGradofMolarDens01}
\delta_x^c[c]_{i,j+\frac{1}{2}}= 0, ~~i\in\{0,N\},~0\leq j\leq M-1,
  \end{equation}
\begin{equation}\label{eqFullyDiscreteGradofMolarDens02}
\delta_y^c[c]_{i+\frac{1}{2},j}= 0, ~~j\in\{0,M\},~0\leq i\leq N-1.
  \end{equation}
 \end{subequations}
We   introduce  the subsets of  $\mathcal{V}_u$ and $\mathcal{V}_v$ involving   the boundary condition  as 
\begin{equation}
\mathcal{V}^0_u=\big\{u\in\mathcal{V}_u| ~u_{0,j+\frac{1}{2}}=u_{N,j+\frac{1}{2}}=0,~~0\leq j\leq M-1\big\},
  \end{equation}
\begin{equation}
\mathcal{V}^0_v=\big\{v\in\mathcal{V}_v|~v_{i+\frac{1}{2},0}=v_{i+\frac{1}{2},M}=0,~~0\leq i\leq N-1\big\}.
  \end{equation}
  Apparently, $\delta_x^c[c]\in \mathcal{V}_u^0$ and $\delta_y^c[c]\in \mathcal{V}_v^0$. 
  
The  difference   operators for   $u\in \mathcal{V}_u $ and $v\in \mathcal{V}_v $ are defined   as
\begin{equation}\label{eqFullyDiscreteGradofMolarDens01}
\delta_x^u[u]_{i+\frac{1}{2},j+\frac{1}{2}}=\frac{u_{i+1,j+\frac{1}{2}}-u_{i,j+\frac{1}{2}}}{h},
\end{equation}
\begin{equation}\label{eqFullyDiscreteGradofMolarDens02}
\delta_y^v[v]_{i+\frac{1}{2},j+\frac{1}{2}}=\frac{v_{i+\frac{1}{2},j+1}-v_{i+\frac{1}{2},j}}{h},
  \end{equation}
  where $0\leq i\leq N-1, ~0\leq j\leq M-1$.
  
    We define the following  discrete inner-products:
\begin{equation*}
\langle c,c'\rangle=h^2\sum_{i=0}^{N-1}\sum_{j=0}^{M-1}c_{i+\frac{1}{2},j+\frac{1}{2}}c'_{i+\frac{1}{2},j+\frac{1}{2}},
~~~c,c' \in\mathcal{V}_c,
\end{equation*}
\begin{equation*}
\langle u,u'\rangle=h^2\sum_{i=1}^{N-1}\sum_{j=0}^{M-1}u_{i,j+\frac{1}{2}}u'_{i,j+\frac{1}{2}},
~~~u,u'\in \mathcal{V}_u^0,
\end{equation*}
\begin{equation*}
\langle v,v'\rangle=h^2\sum_{i=0}^{N-1}\sum_{j=1}^{M-1}v_{i+\frac{1}{2},j}v'_{i+\frac{1}{2},j},~~~v,v'\in \mathcal{V}_v^0.
\end{equation*}
The discrete norms for $c \in\mathcal{V}_c$, $u \in\mathcal{V}_u^0$ and $v \in\mathcal{V}_v^0$ are denoted as
\begin{equation*}
\|c\|=\langle c,c\rangle,~~\|u\|=\langle u,u\rangle,~~\|v\|=\langle v,v\rangle.
\end{equation*}
  
  The semi-implicit fully discrete    scheme  
 is stated: given   $c^{n}\in \mathcal{V}_c$, find   $c^{n+1}\in \mathcal{V}_c$ such that
 \begin{subequations}\label{eqFullyDiscreteSchm}
\begin{equation}\label{eqFullyDiscreteSchmA}
 \frac{ c^{n+1}-n^{n}}{\tau}
 -\kappa\(\delta_x^u[\delta_x^c[c^{n+1}]]+\delta_y^u[\delta_y^c[c^{n+1}]]\)+\nu(c^n)c^{n+1}=s_r(c^n)+\mu_e^{n+1}
 \end{equation}
\begin{equation}\label{eqFullyDiscreteSchmB}
\langle c^{n+1},1\rangle=c_t,
\end{equation}
 \end{subequations}
 where the formulations of functions $\nu(c)$ and $s_r(c)$ are given in \eqref{eqDiscreteEqnsLinTerm} and \eqref{eqDiscreteEqnsSourceTerm} respectively.
 We note that the boundary condition has already been considered in the definitions of  the operators $\delta_x^c$ and $\delta_y^c$ in \eqref{eqFullyDiscreteGradofMolarDens}.
 
 The following summation-by-parts formulas are derived by direct calculations \cite{Wise2009Convex,Wise2010CahnHilliardHeleShaw,shen2016JCP}
\begin{equation}\label{eqFullyDiscreteVariationalPrinciples01}
\langle u,\delta_x^c [c]\rangle=-\langle \delta_x^u [u],c\rangle, ~~u\in \mathcal{V}_u^0,~c\in \mathcal{V}_c,
\end{equation}
\begin{equation}\label{eqFullyDiscreteVariationalPrinciples02}
\langle v,\delta_y^c [c]\rangle=-\langle \delta_y^v [v],c\rangle,~~v\in \mathcal{V}_v^0,~c\in \mathcal{V}_c.
\end{equation}

 \subsection{Unique  solvability}

We first demonstrate the unique  solvability of the linear system \eqref{eqFullyDiscreteSchm}.
\begin{thm}\label{thmSolExistenceOfFullyDiscreteSchm}
Assume  that $c^n\in  \mathcal{V}_c$ satisfies   \eqref{eqMolDensRange} and $\lambda$ is taken to satisfy 
\eqref{eqEnergyFactorizationRepulsionConcaveCondition}.  There exists a unique      $c^{n+1}\in  \mathcal{V}_c$ such that
\eqref{eqFullyDiscreteSchm} holds. 
\end{thm}

\begin{proof}
It suffices to  prove the following homogeneous problem 
has a unique zero solution in $\mathcal{V}_c$
\begin{subequations}\label{thmSolExistenceOfFullyDiscreteSchmProof01}
\begin{equation}\label{thmSolExistenceOfFullyDiscreteSchmProof01A}
 \frac{ c}{\tau}
 -\kappa\(\delta_x^u[\delta_x^c[c]]+\delta_y^u[\delta_y^c[c]]\)+\nu(c^n)c=\mu_e,
 \end{equation}
\begin{equation}\label{thmSolExistenceOfFullyDiscreteSchmProof01B}
\langle c,1\rangle=0.
\end{equation}
 \end{subequations}
Suppose that there exists a nonzero solution $c\in\mathcal{V}_c$ satisfying  \eqref{thmSolExistenceOfFullyDiscreteSchmProof01}. We take the inner product of  \eqref{thmSolExistenceOfFullyDiscreteSchmProof01} with $c$ 
\begin{equation}\label{thmSolExistenceOfFullyDiscreteSchmProof02}
 \frac{ 1}{\tau}\|c\|^2 
 -\kappa\langle\delta_x^u[\delta_x^c[c]],c\rangle-\kappa\langle\delta_y^u[\delta_y^c[c]],c\rangle
 +\langle\nu(c^n)c,c\rangle=\langle\mu_e,c\rangle,
 \end{equation}
 which can be further reduced using \eqref{eqFullyDiscreteVariationalPrinciples01}, \eqref{eqFullyDiscreteVariationalPrinciples02} and \eqref{thmSolExistenceOfFullyDiscreteSchmProof01B} into
\begin{equation}\label{eqSemiDiscreteEqnsUniqueProof03}
\frac{1}{\tau}\|c\|^2+\kappa \(\|\delta_x^c[c]\|^2+\|\delta_y^c[c]\|^2\)+\nu(c_M)\|c\|^2\leq0.
\end{equation}
Since $\nu(c_M)>0$, we concluded from \eqref{eqSemiDiscreteEqnsUniqueProof03} that  $c\equiv0$. This yields a contradiction, and thus, \eqref{eqFullyDiscreteSchm} is    uniquely  solvable.
\end{proof}

\subsection{Discrete maximum principle}
We are going to prove the discrete maximum principle of molar density. The following lemma is an essential ingredient  of the   proof.
\begin{lem}\label{lemDiscreteLapalace}
Let   $c^-=\min(c-c_m,0)$ and $c^+=\max(c-c_M,0)$, where $c\in \mathcal{V}_c$ and $c_m<c_M$. Then we have
\begin{equation}\label{eqDiscreteLapalace01A}
\langle \delta_x^c [c^-],\delta_x^c [c^-]\rangle\leq-\langle \delta_x^u [\delta_x^c [c]],c^-\rangle,
\end{equation}
\begin{equation}\label{eqDiscreteLapalace01C}
\langle \delta_x^c [c^+],\delta_x^c [c^+]\rangle\leq-\langle \delta_x^u [\delta_x^c [c]],c^+\rangle,
\end{equation}
\begin{equation}\label{eqDiscreteLapalace01B}
\langle \delta_y^c [c^-],\delta_y^c [c^-]\rangle\leq-\langle \delta_y^v [\delta_y^c [c]],c^-\rangle,
\end{equation}
\begin{equation}\label{eqDiscreteLapalace01D}
\langle \delta_y^c [c^+],\delta_y^c [c^+]\rangle\leq-\langle \delta_y^v [\delta_y^c [c]],c^+\rangle.
\end{equation}
\end{lem}
\begin{proof}
Let $a$ and $b$ be two real scalar numbers. We further define  $a^-=\min(a-c_m,0)$ and $b^-=\min(b-c_m,0)$.
It is  apparent that
\begin{equation}\label{eqDiscreteLapalaceProof01A}
\(a-c_m\)a^-
=  |a^-|^2,~~~\(b-c_m\)b^-
=  |b^-|^2.
 \end{equation}
Since $a-c_m\geq a^-$ and $b^-\leq0$, we get
 \begin{equation}\label{eqDiscreteLapalaceProof01B}
\(a-c_m\)b^-
\leq a^-b^-.
 \end{equation}
Applying \eqref{eqFullyDiscreteVariationalPrinciples01}, \eqref{eqDiscreteLapalaceProof01A} and \eqref{eqDiscreteLapalaceProof01B}, we can derive the inequality \eqref{eqDiscreteLapalace01A} as follows 
\begin{align}\label{eqDiscreteLapalaceProof02}
-\langle \delta_x^u [\delta_x^c [c]],c^-\rangle&=\langle \delta_x^c [c],\delta_x^c [c^-]\rangle\nonumber\\
&=h^2\sum_{i=1}^{N-1}\sum_{j=0}^{M-1}\delta_x^c [c]_{i,j+\frac{1}{2}}\delta_x^c [c^-]_{i,j+\frac{1}{2}}\nonumber\\
&=\sum_{i=1}^{N-1}\sum_{j=0}^{M-1}\(c_{i+\frac{1}{2},j+\frac{1}{2}}-c_{i-\frac{1}{2},j+\frac{1}{2}}\)\(c^-_{i+\frac{1}{2},j+\frac{1}{2}}-c^-_{i-\frac{1}{2},j+\frac{1}{2}}\)\nonumber\\
&\geq\sum_{i=1}^{N-1}\sum_{j=0}^{M-1}\(|c^-_{i+\frac{1}{2},j+\frac{1}{2}}|^2+|c^-_{i-\frac{1}{2},j+\frac{1}{2}}|^2
-2c^-_{i+\frac{1}{2},j+\frac{1}{2}}c^-_{i-\frac{1}{2},j+\frac{1}{2}}\)
\nonumber\\
&=\sum_{i=1}^{N-1}\sum_{j=0}^{M-1}\(c^-_{i+\frac{1}{2},j+\frac{1}{2}}-c^-_{i-\frac{1}{2},j+\frac{1}{2}}\)^2\nonumber\\
&=\langle \delta_x^c [c^-],\delta_x^c [c^-]\rangle.
\end{align}

We now prove the inequality \eqref{eqDiscreteLapalace01C}. 
Let  $a^+=\max(a-c_M,0)$ and $b^+=\max(b-c_M,0)$, then we have
\begin{equation}\label{eqDiscreteLapalaceProof03}
\(a-c_M\)a^+
=  |a^+|^2,~~~\(b-c_M\)b^+
=  |b^+|^2.
 \end{equation}
Since $a-c_M\leq a^+$ and $b^+\geq0$, we get
 \begin{equation}\label{eqDiscreteLapalaceProof04}
\(a-c_M\)b^+
\leq a^+b^+.
 \end{equation}
Applying \eqref{eqFullyDiscreteVariationalPrinciples01}, \eqref{eqDiscreteLapalaceProof03} and \eqref{eqDiscreteLapalaceProof04}, we deduce the inequality \eqref{eqDiscreteLapalace01C} as     
\begin{align}\label{eqDiscreteLapalaceProof05}
-\langle \delta_x^u [\delta_x^c [c]],c^+\rangle&=\langle \delta_x^c [c],\delta_x^c [c^+]\rangle\nonumber\\
&=h^2\sum_{i=1}^{N-1}\sum_{j=0}^{M-1}\delta_x^c [c]_{i,j+\frac{1}{2}}\delta_x^c [c^+]_{i,j+\frac{1}{2}}\nonumber\\
&=\sum_{i=1}^{N-1}\sum_{j=0}^{M-1}\(c_{i+\frac{1}{2},j+\frac{1}{2}}-c_{i-\frac{1}{2},j+\frac{1}{2}}\)\(c^+_{i+\frac{1}{2},j+\frac{1}{2}}-c^+_{i-\frac{1}{2},j+\frac{1}{2}}\)\nonumber\\
&\geq\sum_{i=1}^{N-1}\sum_{j=0}^{M-1}\(|c^+_{i+\frac{1}{2},j+\frac{1}{2}}|^2+|c^+_{i-\frac{1}{2},j+\frac{1}{2}}|^2
-2c^+_{i+\frac{1}{2},j+\frac{1}{2}}c^+_{i-\frac{1}{2},j+\frac{1}{2}}\)
\nonumber\\
&=\langle \delta_x^c [c^+],\delta_x^c [c^+]\rangle.
\end{align}
The rest inequalities \eqref{eqDiscreteLapalace01B} and \eqref{eqDiscreteLapalace01D} can be proved by the similar approaches. 
\end{proof}
 
 We are now ready to prove the maximum principle of the fully discrete scheme.
\begin{thm}\label{thmFullyDiscreteMaximumPrinciple}
Assume  that $c^n\in  \mathcal{V}_c$ satisfies   \eqref{eqMolDensRange} and $\lambda$ is taken to satisfy 
\eqref{eqEnergyFactorizationRepulsionConcaveCondition}.  
Under the condition \eqref{eqMaximumPrincipleCondition}, we have  $c_m\leq c^{n+1}\leq c_M$.
\end{thm}

\begin{proof}
It can be proved using the similar   routines as in the proof of  Theorem \ref{thmMaximumPrinciple}. The major difference is that we use   Lemma \ref{lemDiscreteLapalace} to treat the discrete form of Laplace operator. So we just give a brief proof.

Assuming  that $c_m\leq c^n\leq c_M$   for $n\geq0$,  we will  prove that $c_m\leq c^{n+1}\leq c_M$.
 Let $c_{-}^{n+1}=\min(c^{n+1}-c_m,0)$. It can be derived from  \eqref{eqFullyDiscreteSchmA}   that
 \begin{align}\label{eqFullyDiscreteSchmMaximumProof01}
&\frac{1}{\tau }\langle c^{n+1}-c^n,c_{-}^{n+1}\rangle-\kappa \langle\delta_x^u[\delta_x^c[c^{n+1}]]+\delta_y^u[\delta_y^c[c^{n+1}]],c_{-}^{n+1}\rangle\nonumber\\
&+\langle\nu(c^n)(c^{n+1}-c_m),c_-^{n+1}\rangle
=\langle\mu_e^{n+1},c_-^{n+1}\rangle+\langle s_r(c^n)-c_m\nu(c^n),c_-^{n+1}\rangle.
\end{align}
The second term on the left-hand side of \eqref{eqFullyDiscreteSchmMaximumProof01} can be estimated using     Lemma \ref{lemDiscreteLapalace} 
\begin{align}\label{eqFullyDiscreteSchmMaximumProof03}
-\kappa \langle\delta_x^u[\delta_x^c[c^{n+1}]]+\delta_y^u[\delta_y^c[c^{n+1}]],c_{-}^{n+1}\rangle
\geq\kappa\(\|\delta_x^c[c_-^{n+1}]\|^2+\|\delta_y^c[c_-^{n+1}]\|^2\).
\end{align}
Following  the similar  routines  used to deduce \eqref{eqSemiDiscreteSchmMaximumProof06}, we obtain
\begin{align}\label{eqFullyDiscreteSchmMaximumProof06}
\frac{1}{\tau}\|c_{-}^{n+1}\|^2+\kappa\(\|\delta_x^c[c_-^{n+1}]\|^2+\|\delta_y^c[c_-^{n+1}]\|^2\)+\nu(c_M)\|c_-^{n+1}\|^2
\leq0,
\end{align}
which yields $c_{-}^{n+1}=0$, and thus, $c^{n+1}\geq c_m$.
It is similar to prove $c^{n+1}\leq c_M$.
\end{proof}

\subsection{Energy stability}
We denote the discrete total free energy as
\begin{equation}\label{eqFullydiscreteEnergy}
F_h(c^n)=\langle f_b(c^{n}),1\rangle+\frac{1}{2}\(\|\delta_x^c[c^{n}]\|^2+\|\delta_y^c[c^{n}]\|^2\).
\end{equation}
\begin{thm}\label{thmFullySchmEnergyStability}
Assume  that $c^n$ satisfies   \eqref{eqMolDensRange} and $\lambda$ is taken to satisfy 
\eqref{eqEnergyFactorizationRepulsionConcaveCondition}.  
Under the condition \eqref{eqMaximumPrincipleCondition}, for any time step size $\tau$, we have  
\begin{equation}\label{eqFullySchmEnergyStability}
F_h(c^{n+1})\leq F_h(c^n).
\end{equation}
\end{thm}
\begin{proof}
Using \eqref{eqFullyDiscreteVariationalPrinciples01},   we can derive that
\begin{align}\label{eqFullySchmEnergyStabilityProof01}
\frac{1}{2}\(\|\delta_x^c[c^{n+1}]\|^2-\|\delta_x^c[c^{n}]\|^2\)
&=\frac{1}{2}\(\langle \delta_x^c [c^{n+1}],\delta_x^c [c^{n+1}]\rangle
-\langle \delta_x^c [c^{n}],\delta_x^c [c^{n}]\rangle\)\nonumber\\
&=\langle \delta_x^c [c^{n+1}],\delta_x^c [c^{n+1}]-\delta_x^c [c^{n}]\rangle
-\frac{1}{2}\| \delta_x^c [c^{n+1}]-\delta_x^c [c^{n}]\|^2\nonumber\\
&\leq\langle \delta_x^c [c^{n+1}],\delta_x^c [c^{n+1}-c^{n}]\rangle\nonumber\\
&=-\langle \delta_x^u [\delta_x^c [c^{n+1}]], c^{n+1}-c^{n}\rangle.
\end{align}
It is similar to deduce that
\begin{align}\label{eqFullySchmEnergyStabilityProof02}
\frac{1}{2}\(\|\delta_y^c[c^{n+1}]\|^2-\|\delta_y^c[c^{n}]\|^2\)
\leq -\langle \delta_y^u [\delta_y^c [c^{n+1}]], c^{n+1}-c^{n}\rangle.
\end{align}
Using similar   arguments in \eqref{eqSemiSchmEnergyStabilityProof01} and the estimates  \eqref{eqFullySchmEnergyStabilityProof01} and \eqref{eqFullySchmEnergyStabilityProof01}, we obtain 
\begin{align}\label{eqFullySchmEnergyStabilityProof03}
\langle f_b(c^{n+1})-f_b(c^{n}),1\rangle+\frac{1}{2}\(\|\delta_x^c[c^{n+1}]\|^2-\|\delta_x^c[c^{n}]\|^2\)\nonumber\\
+\frac{1}{2}\(\|\delta_y^c[c^{n+1}]\|^2-\|\delta_y^c[c^{n}]\|^2\)
\leq-\frac{1}{\tau}\|c^{n+1}-c^{n}\|^2,
\end{align}
which yields \eqref{eqFullySchmEnergyStability}.
\end{proof}

\section{Numerical results}

In this section,  we present some numerical results to show the performance of the proposed method and  verify the theoretical analysis. We  consider a hydrocarbon  substance,  n-butane (nC$_4$). The related physical data is provided  in Table \ref{tabParametersPREOS}.  The temperature is fixed at 330 K. The value of $\vartheta_0$ has no effects  for the isothermal fluids, so we take $\vartheta_0=0$.  In    numerical tests,   we use the   gas molar density $c^G = 249.1123~\textnormal{mol/m}^3$ and the liquid molar density $c^L = 9526.8428~\textnormal{mol/m}^3$
to initialize   the  distribution of molar density. We note that mass transfer between two phases may take place in dynamical process; that is, the gas phase may be condensed  into the liquid phase, while the liquid phase may also be  vaporized into the gas phase. As a result,  we take $c_m=0.9c^G$ and $c_M=1.1c^L$  in \eqref{eqMolDensRange} and calculate $\epsilon_0=\beta c_M=0.7585$. To satisfy the condition \eqref{eqEnergyFactorizationRepulsionConcaveCondition}, we use the value of $\epsilon_0$ to calculate $\lambda$ as follows
\begin{equation}\label{eqEnergyFactorizationRepulsionConcaveConditionCase01}
\lambda=     \frac{ \epsilon_0}{(1-\epsilon_0)^2}+\(   \frac{\epsilon_0^2}{(1-\epsilon_0)^4}-2 \ln\(1-\epsilon_0\)\frac{\epsilon_0}{(1-\epsilon_0)^2}\)^{1/2}=27.3656.
\end{equation}

 \begin{table} 
\caption{Physical parameters of nC$_4$}
\begin{center}
\begin{tabular}{ccccccccc}
\hline
  $P_c$(bar)   ~~~~~~&  $T_c$(K)   ~~~~~~& $\omega$    \\
\hline
  38.0             ~~~~~~& 425.2          ~~~~~~&0.199           \\
\hline
\end{tabular}
\end{center}
\label{tabParametersPREOS}
\end{table}

In all numerical tests, the spatial domain is taken as $\Omega=[-L,L]^2$, where $L=15$ nm, and    a  uniform rectangular mesh with   $100\times100$ elements is employed. The proposed method admits the use of a very time step size,  so we take $\tau=10^{10}$ s.

 In this example,  we simulate the droplet shrinking problem for 200 time steps. 
In Figure \ref{SquarenC4MolarDensityOfnC4Temperature330K}, we show   molar density distributions  after different time steps, which demonstrate that the square  droplet is changing   to a circle shape.  In Figures \ref{SquarenC4FreeEnergyOfnC4Temperature330K}, we also illustrate the distributions of the bulk free energy density and gradient free energy density  after 200 time steps, respectively.  

In Figure \ref{SquarenC4equiChPtlSeqTemperature330}, the lower bound $\underline{\mu}$ and the upper bound $\bar{\mu}$  are calculated  as
$$\underline{\mu}=\max_{c_m\leq c\leq c_M}\(c_m\nu(c)-s_r(c)\),~~~\bar{\mu}=\min_{c_m\leq c\leq c_M}\(c_M\nu(c)-s_r(c)\).$$
 Figure \ref{SquarenC4equiChPtlSeqTemperature330} shows that $\mu_e^n$ always varies between $\underline{\mu}$ and $\bar{\mu}$, thus the condition \eqref{eqMaximumPrincipleCondition}  is invariably true. Figures \ref{SquarenC4MolDensMinBoundTemperature330} and \ref{SquarenC4MolDensMaxBoundTemperature330} depict the minimum and maximum values of molar density $c^n$.  Due to the effects of liquid vaporization and gas condensation, molar density changes a lot  at the beginning  period and then tends to the stead values in the later period, but we observe that  $c^n$ always fluctuates between $c_m$ and $c_M$, thus the maximum principle is verified.

 Figure \ref{SquarenC4freeEnergyFullTemperature330} illustrates that the total energies are dissipated with time steps, while Figure \ref{SquarenC4freeEnergyPartTemperature330}   plots total energies at the last twenty time steps, which are still decreasing. Therefore, the proposed scheme can preserve the energy dissipation law.

\begin{figure}
            \centering \subfigure[Initial]{
            \begin{minipage}[b]{0.31\textwidth}
               \centering
             \includegraphics[width=0.95\textwidth,height=1.33in]{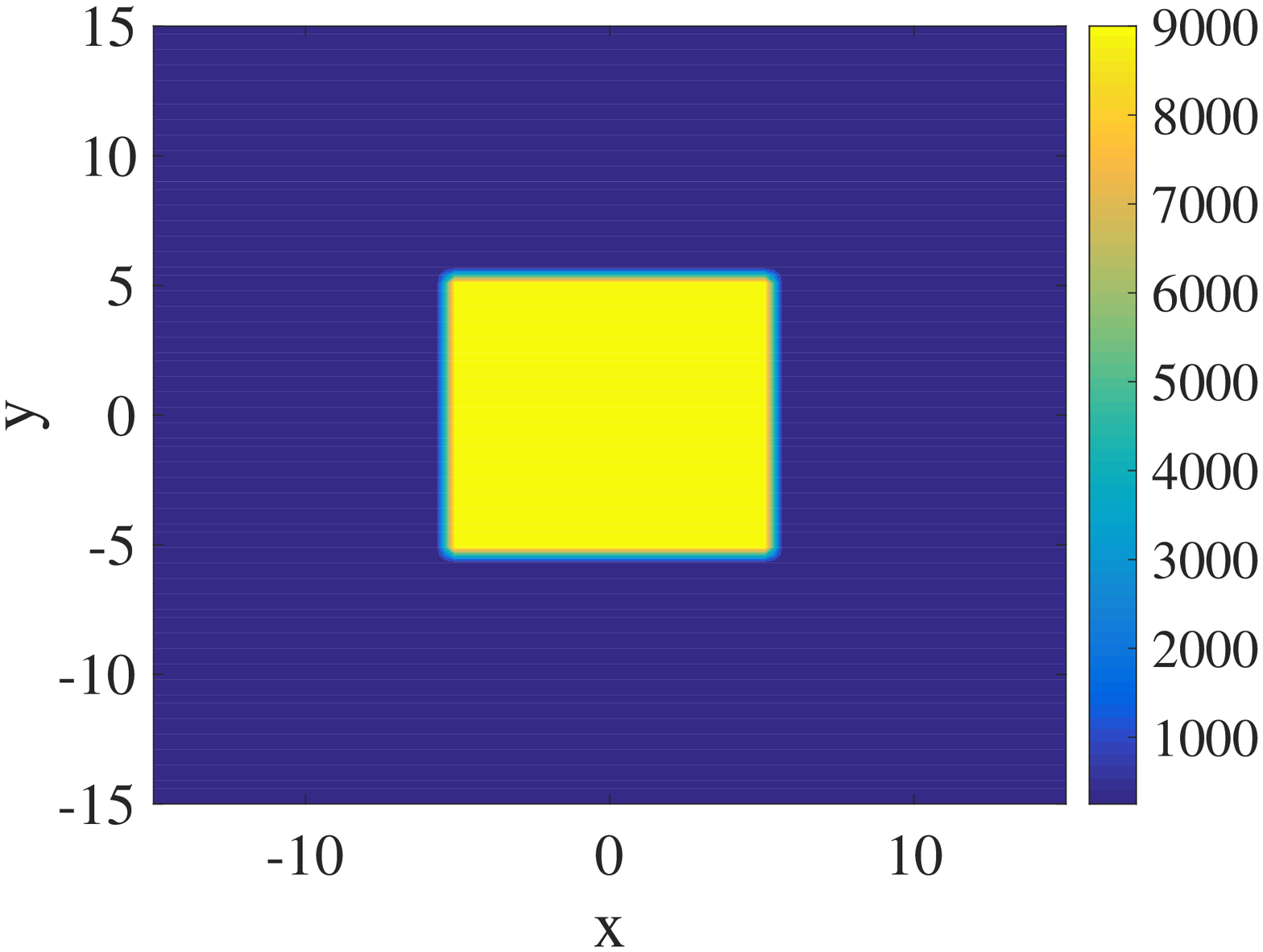}
              \label{IsolatedSysnC4MolarDensityOfnC4iT0}
            \end{minipage}
            }
            \centering \subfigure[$n=100$]{
            \begin{minipage}[b]{0.31\textwidth}
            \centering
             \includegraphics[width=0.95\textwidth,height=1.33in]{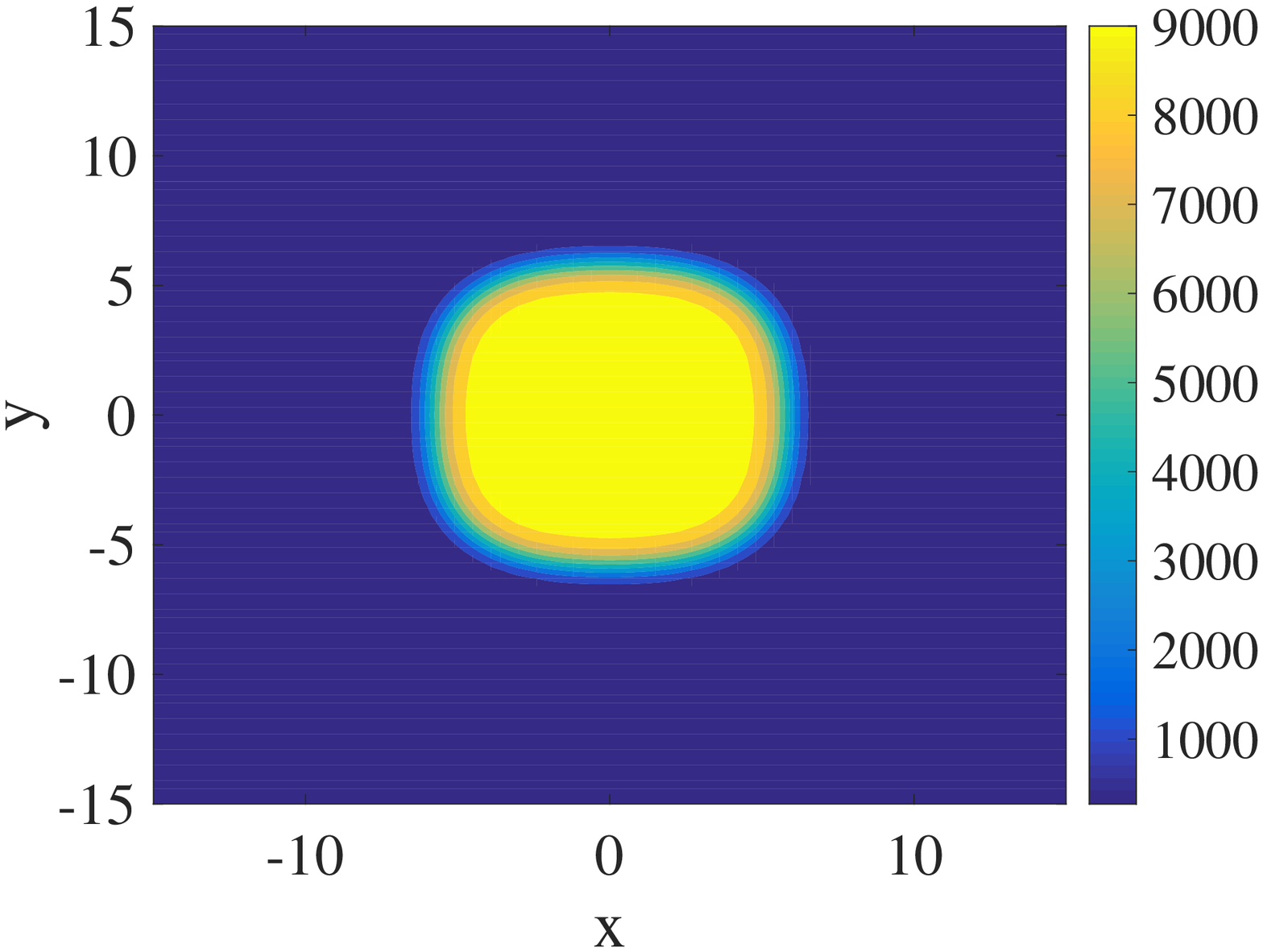}
            \end{minipage}
            }
           \centering \subfigure[$n=200$]{
            \begin{minipage}[b]{0.3\textwidth}
            \centering
             \includegraphics[width=0.95\textwidth,height=1.33in]{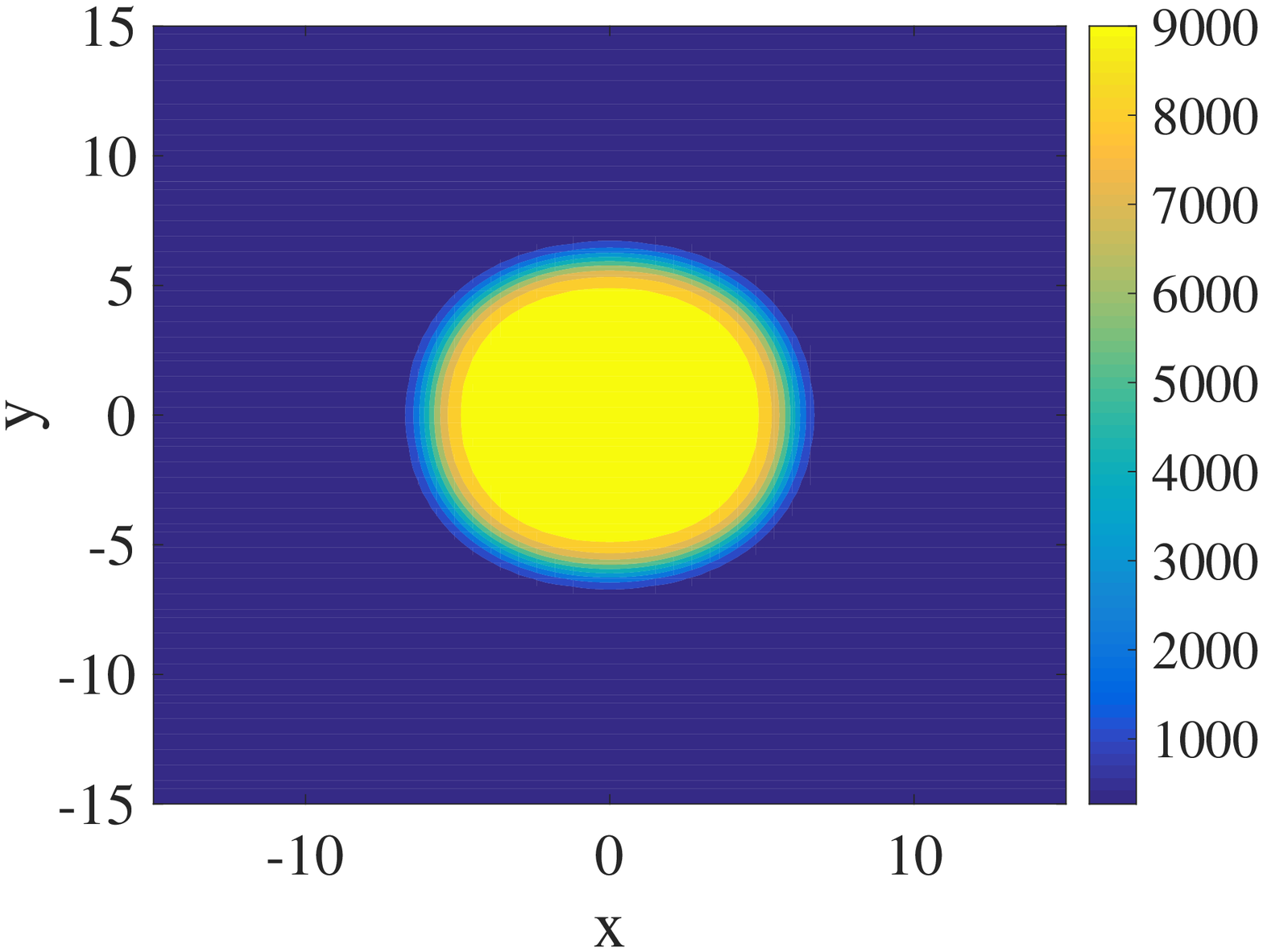}
            \end{minipage}
            }
           \caption{Example 1:     molar density distributions.}
            \label{SquarenC4MolarDensityOfnC4Temperature330K}
 \end{figure}

\begin{figure}
            \centering \subfigure[]{
            \begin{minipage}[b]{0.35\textwidth}
            \centering
             \includegraphics[width=\textwidth,height=0.9\textwidth]{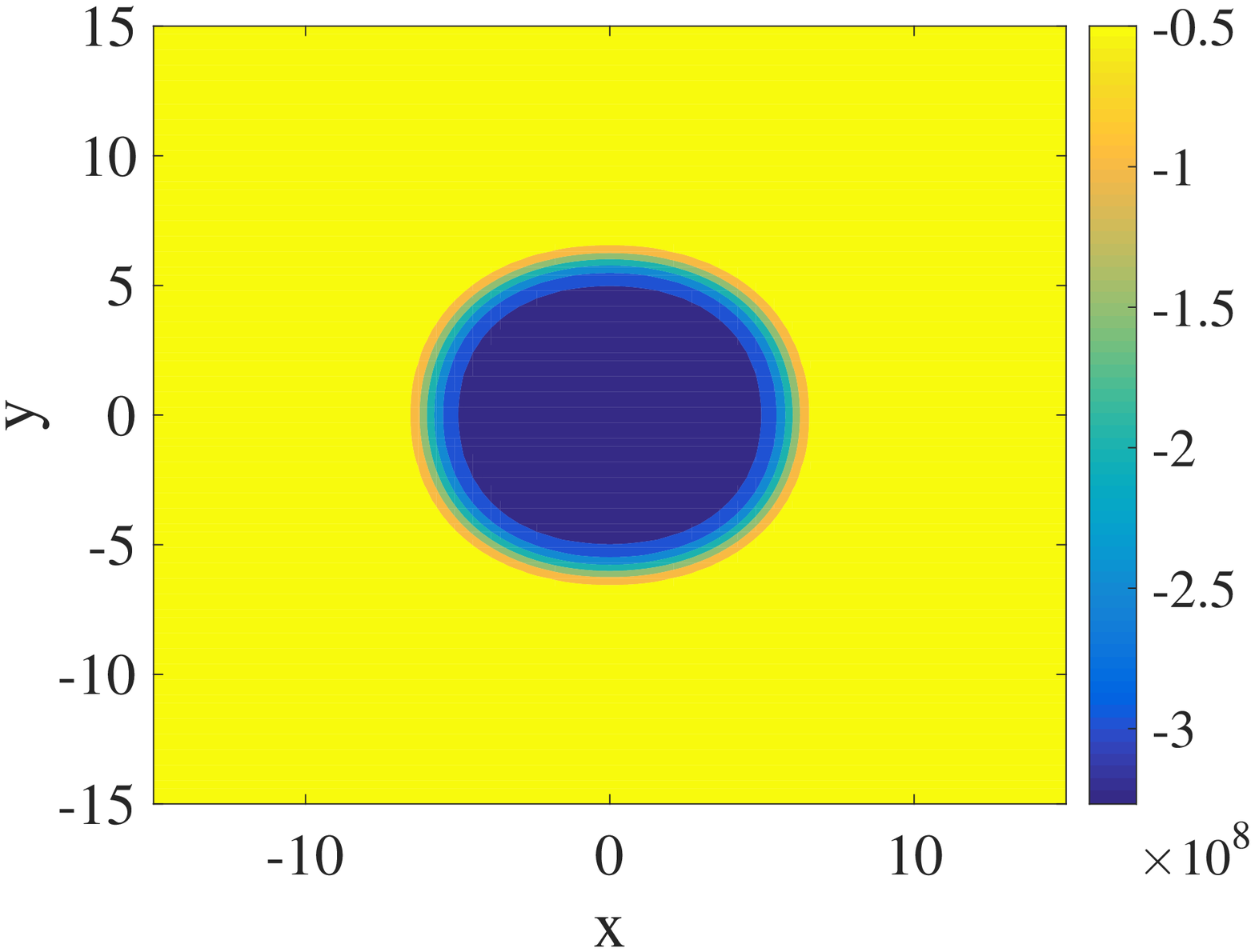}
            \end{minipage}
            }
           \centering \subfigure[]{
            \begin{minipage}[b]{0.35\textwidth}
            \centering
             \includegraphics[width=\textwidth,height=0.9\textwidth]{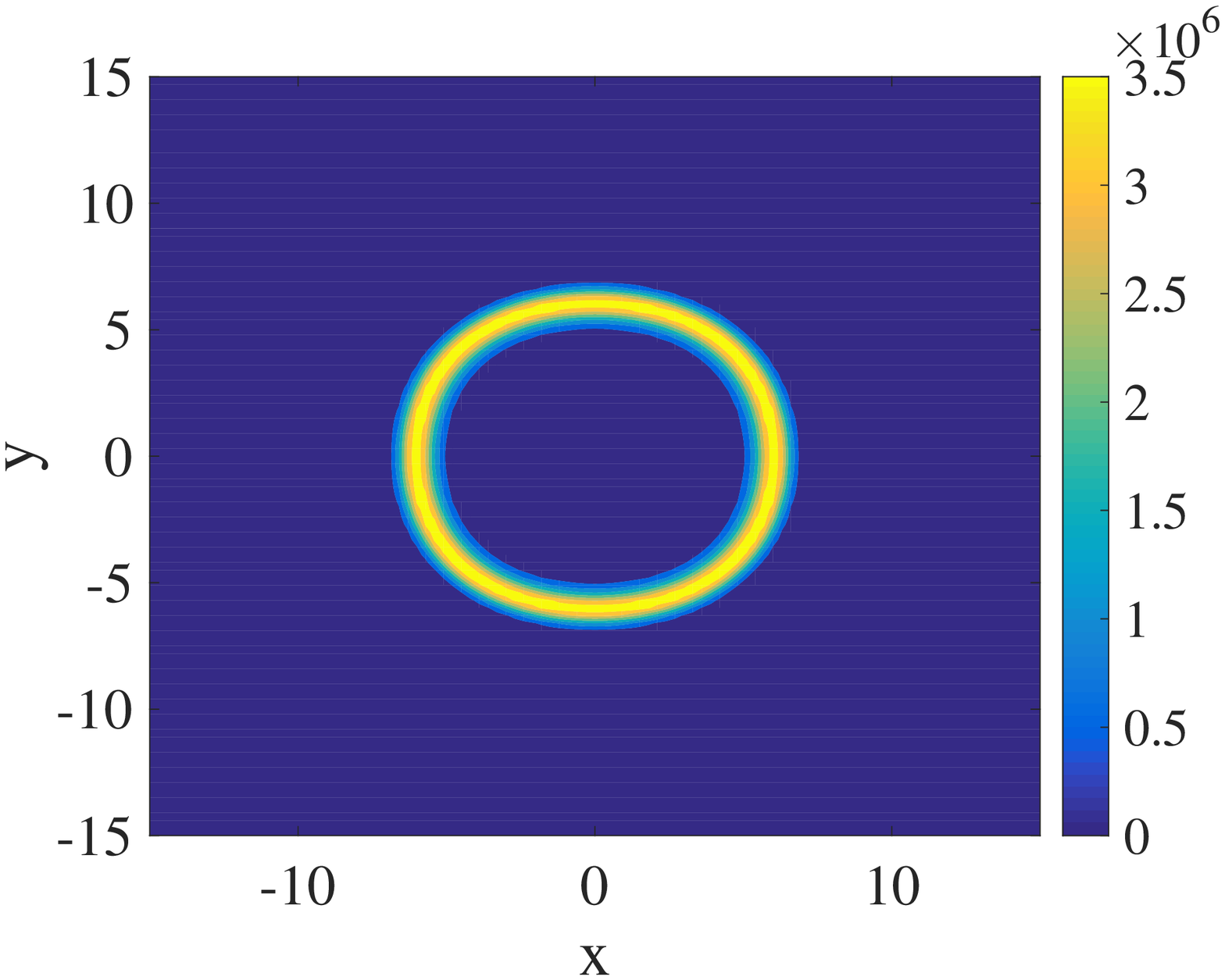}
            \end{minipage}
            }
           \caption{Example 1:  (a) the  bulk free energy density   after      200   time steps; (b) the gradient free energy  after      200   time steps.}
            \label{SquarenC4FreeEnergyOfnC4Temperature330K}
 \end{figure}

\begin{figure}
           \centering \subfigure[Condition \eqref{eqMaximumPrincipleCondition}]{
            \begin{minipage}[b]{0.3\textwidth}
            \centering
             \includegraphics[width= \textwidth,height=0.95\textwidth]{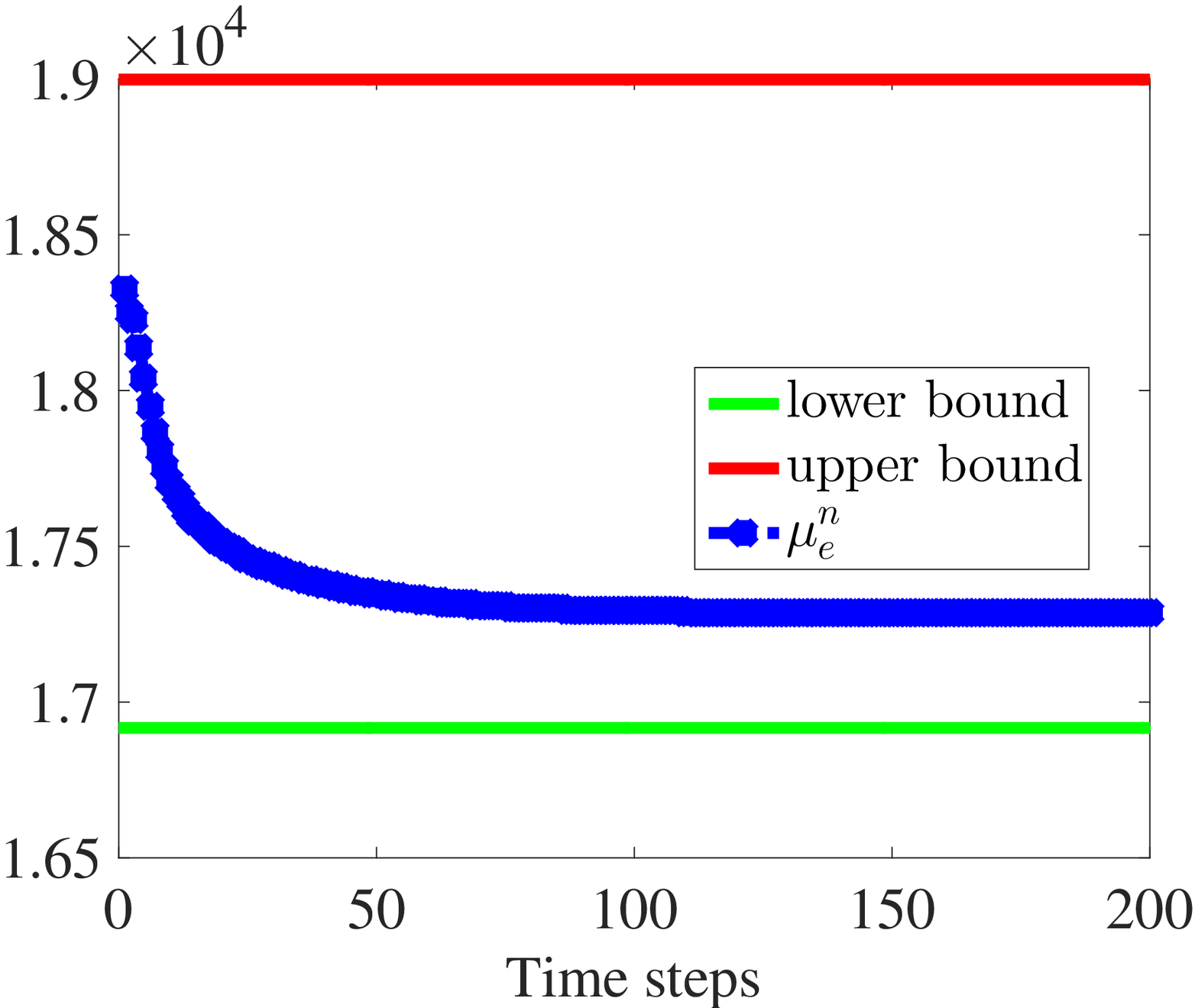}
             \label{SquarenC4equiChPtlSeqTemperature330}
            \end{minipage}
            }
            \centering \subfigure[Minima of molar density]{
            \begin{minipage}[b]{0.3\textwidth}
            \centering
             \includegraphics[width= \textwidth,height=0.95\textwidth]{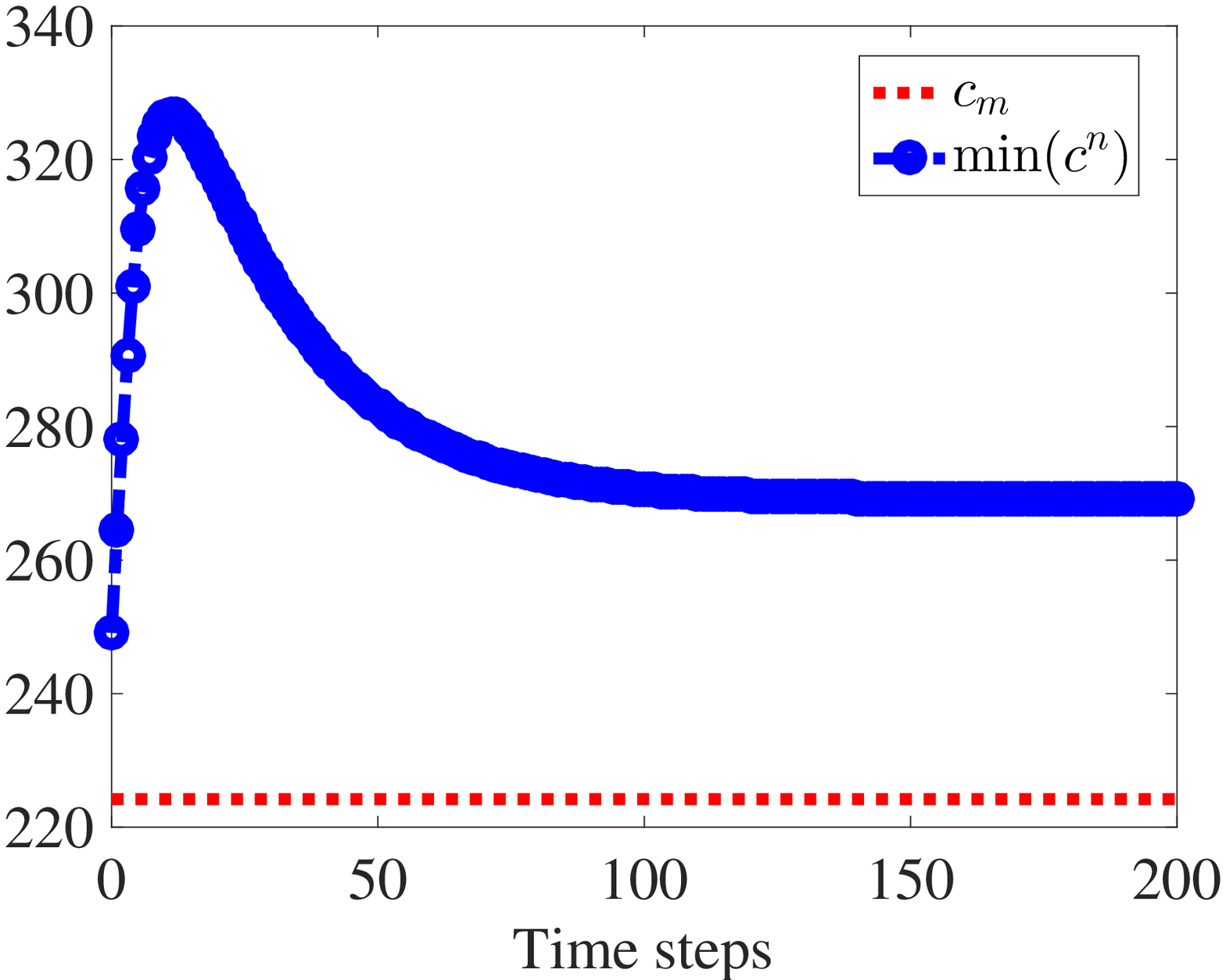}
             \label{SquarenC4MolDensMinBoundTemperature330}
            \end{minipage}
            }
            \centering \subfigure[Maxima of molar density]{
            \begin{minipage}[b]{0.3\textwidth}
            \centering
             \includegraphics[width= \textwidth,height= 0.95\textwidth]{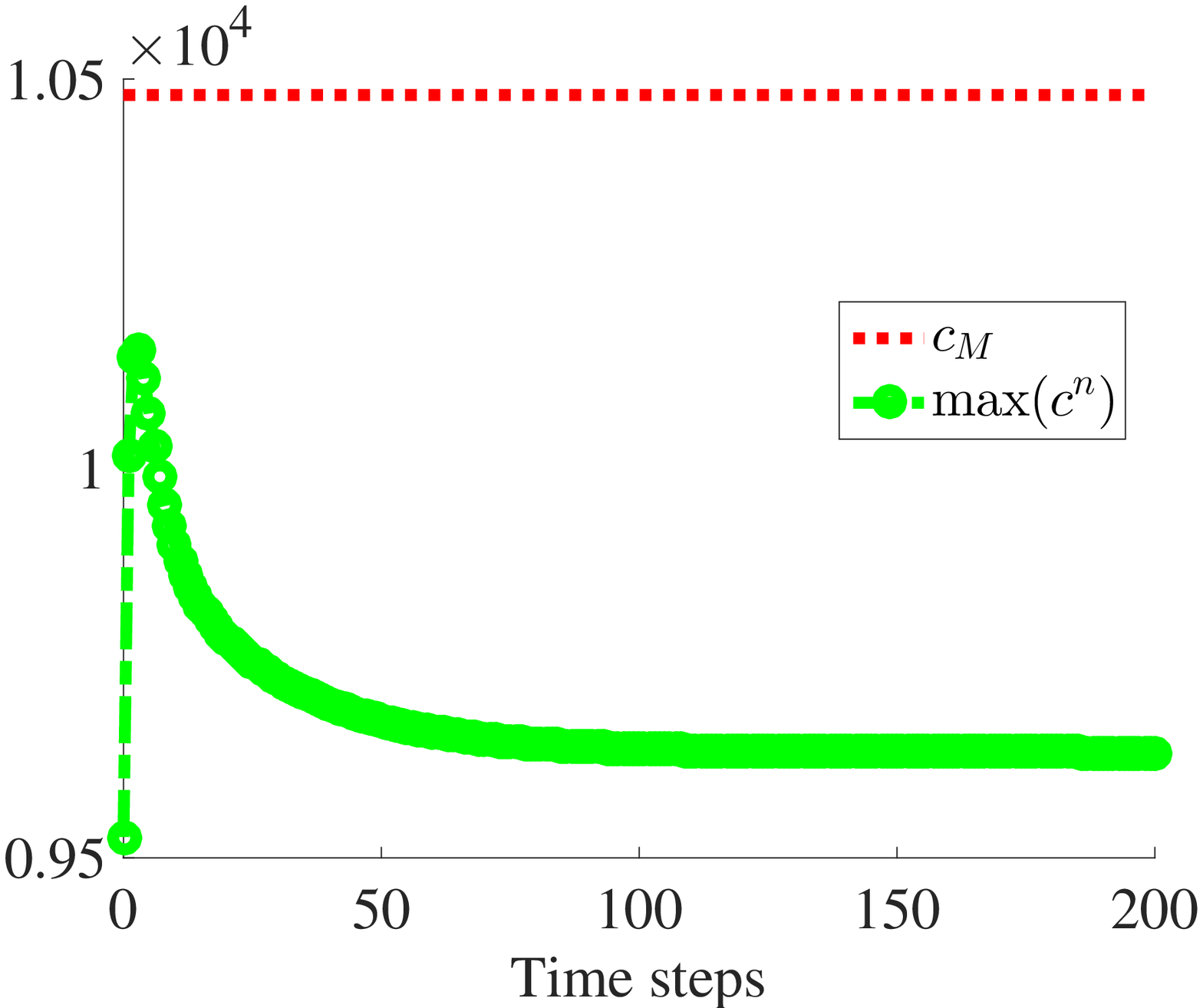}
             \label{SquarenC4MolDensMaxBoundTemperature330}
            \end{minipage}
            }
           \caption{Example 1: verification of the condition \eqref{eqMaximumPrincipleCondition} and the maximum principle.}
            \label{ex1verificationcondition}
 \end{figure}
 
 \begin{figure}
           \centering \subfigure[]{
            \begin{minipage}[b]{0.38\textwidth}
            \centering
             \includegraphics[width=0.95\textwidth,height=1.33in]{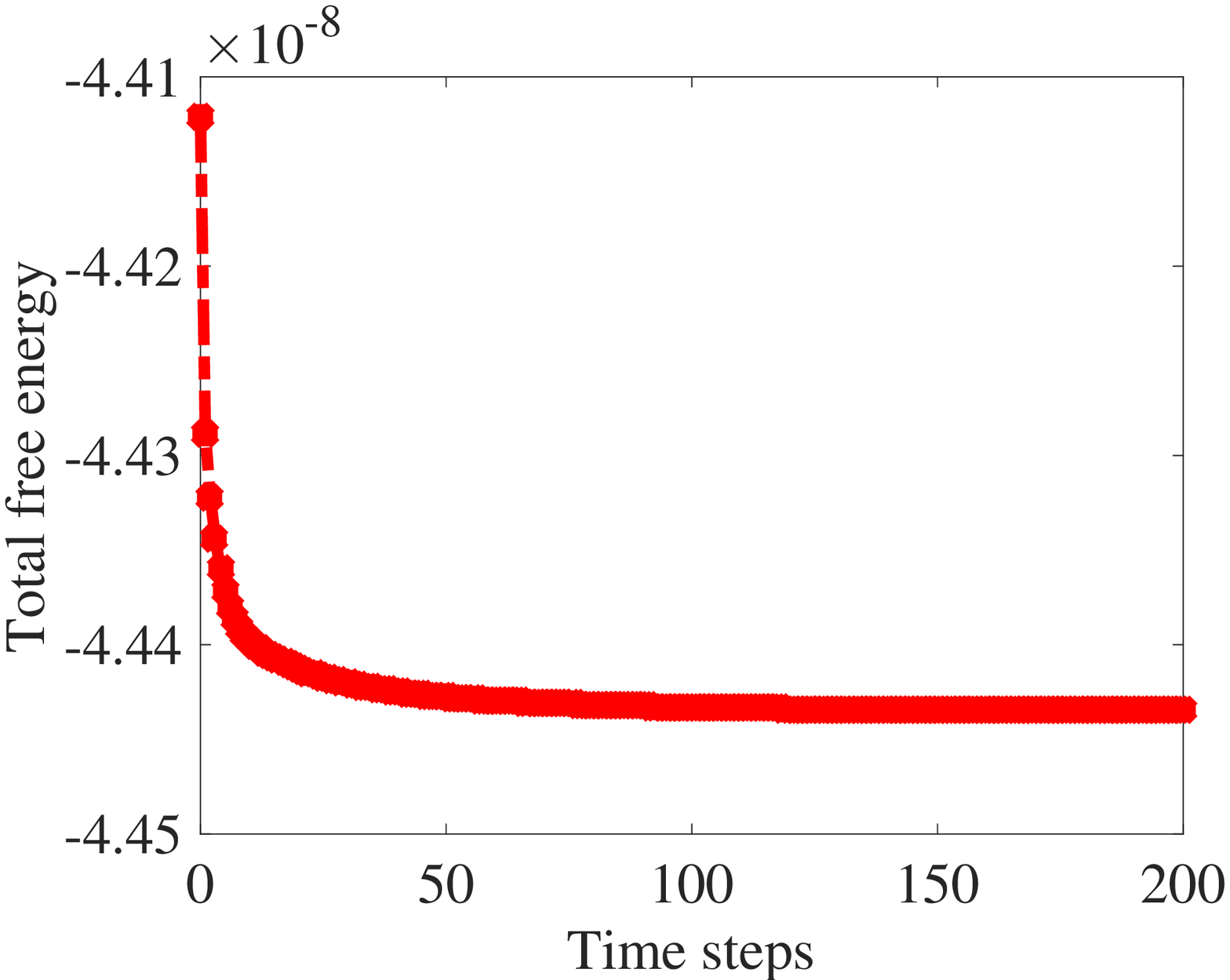}
             \label{SquarenC4freeEnergyFullTemperature330}
            \end{minipage}
            }
            \centering \subfigure[]{
            \begin{minipage}[b]{0.38\textwidth}
            \centering
             \includegraphics[width=0.95\textwidth,height=1.33in]{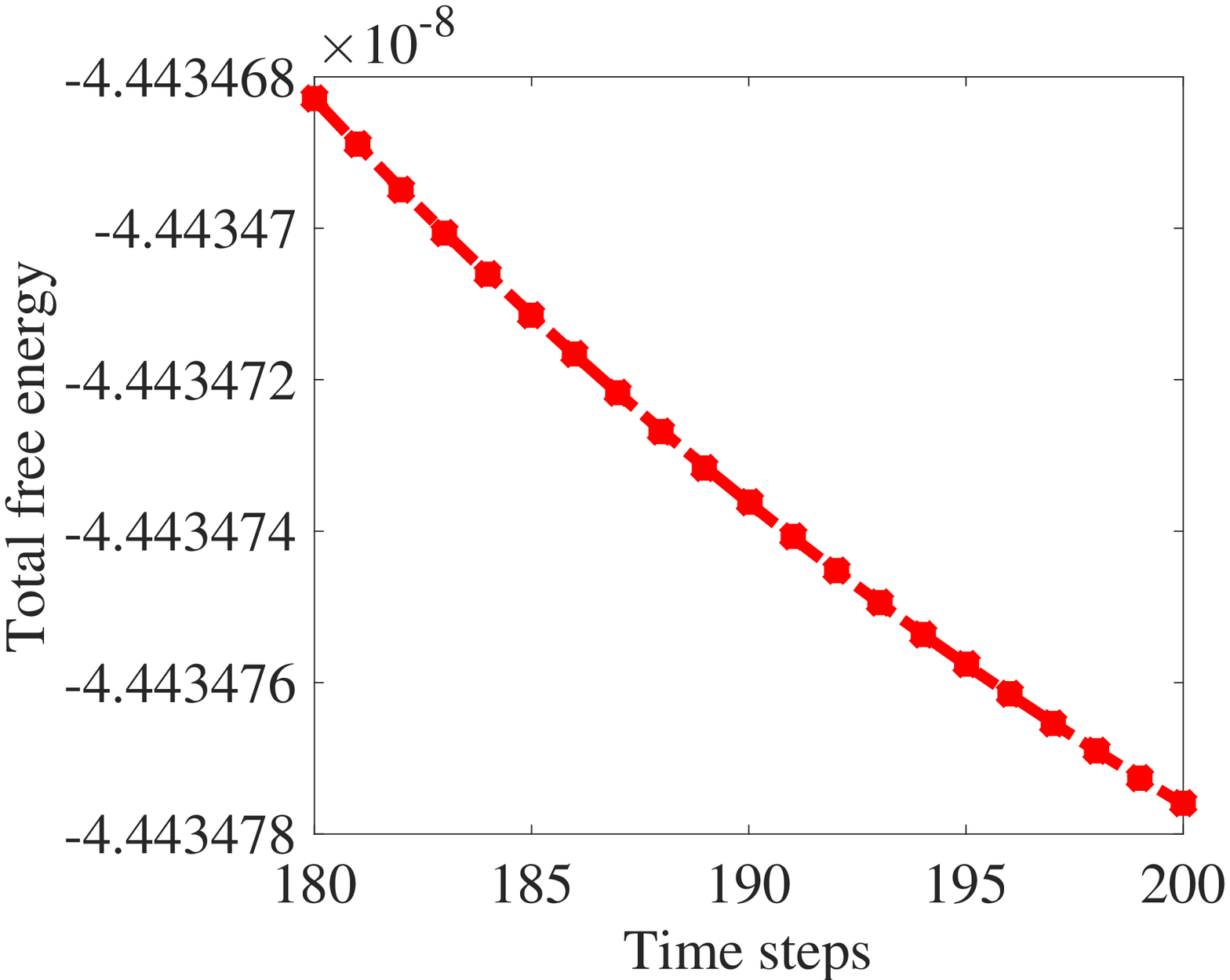}
             \label{SquarenC4freeEnergyPartTemperature330}
            \end{minipage}
            }
           \caption{Example 1: total energy profiles with time steps.}
            \label{SquarenC4freeEnergyTemperature330}
 \end{figure}

\section{Conclusions}


 A novel energy factorization (EF) approach has been proposed to construct the linear energy stable numerical scheme for the diffuse interface model with the Peng-Robinson equation of state that is   one of the most  useful and prominent     tools    in   chemical engineering and petroleum industry. 
 Compared with the convex-splitting schemes, the   semi-implicit numerical scheme constructed by the EF approach   is linear and  easy-to-implement.  Compared  with the IEQ/SAV schemes,  the proposed scheme   inherits the original energy dissipation law.  Moreover, we prove that    the maximum principle  holds for both the time sem-discrete   form and the cell-centered finite difference fully discrete form under certain conditions.  Numerical results   validate  the stability and efficiency of the proposed scheme.  In the future work, we will study the applications of the EF approach to multi-component fluids and phase-field models.

\small

\end{document}